\documentclass[english,british]{amsart}
\usepackage[T1]{fontenc}
\usepackage[latin9]{inputenc}
\usepackage[a4paper]{geometry}
\usepackage{amsthm}
\usepackage{amstext}
\usepackage{amssymb}
\usepackage{setspace}
\makeatletter
\numberwithin{equation}{section}
\numberwithin{figure}{section}
\theoremstyle{plain}
\newtheorem{thm}{\protect\theoremname}[section]
  \theoremstyle{remark}
  \newtheorem*{rem*}{\protect\remarkname}
  \theoremstyle{plain}
  \newtheorem{prop}[thm]{\protect\propositionname}
  \theoremstyle{plain}
  \newtheorem{cor}[thm]{\protect\corollaryname}
  \theoremstyle{definition}
  \newtheorem{defn}[thm]{\protect\definitionname}
  \theoremstyle{remark}
  \newtheorem*{acknowledgement*}{\protect\acknowledgementname}
  \theoremstyle{plain}
  \newtheorem{lem}[thm]{\protect\lemmaname}
  \theoremstyle{remark}
  \newtheorem{rem}[thm]{\protect\remarkname}

\usepackage{amscd}
\usepackage{mathptmx}
\usepackage{bbm}

\newcommand{\e}{\mathrm{e}}
\newcommand{\1}{1}
\newcommand{\N}{\mathbb{N}}

\newcommand{\R}{\mathbb{R}}
\newcommand{\C}{\mathbb{C}}

\renewcommand{\Pi}{\pi}
\renewcommand{\emptyset}{\varnothing}
\renewcommand{\hat}{\widehat}

\newcommand\diam{\mathrm{diam}}

\DeclareMathOperator*{\Int}{Int}

\newcommand\id{\mathrm{id}}
\newcommand\Chat{\hat{\mathbb{C}}}

\DeclareMathOperator*{\card}{card}
\DeclareMathOperator*{\supp}{supp}

\DeclareMathOperator*{\Aut}{Aut}
\DeclareMathOperator*{\CV}{CV}
\DeclareMathOperator*{\Hol}{\textnormal{H\"ol}}
\newcommand{\M}{\mathit{Q}}

\DeclareMathOperator*{\Rat}{Rat}

\usepackage{hyperref}
\@ifundefined{textmu}
 {\usepackage{textcomp}}{}

\makeatother

\usepackage{babel}
  \addto\captionsbritish{\renewcommand{\acknowledgementname}{Acknowledgement}}
  \addto\captionsbritish{\renewcommand{\corollaryname}{Corollary}}
  \addto\captionsbritish{\renewcommand{\definitionname}{Definition}}
  \addto\captionsbritish{\renewcommand{\lemmaname}{Lemma}}
  \addto\captionsbritish{\renewcommand{\propositionname}{Proposition}}
  \addto\captionsbritish{\renewcommand{\remarkname}{Remark}}
  \addto\captionsbritish{\renewcommand{\theoremname}{Theorem}}
  \addto\captionsenglish{\renewcommand{\acknowledgementname}{Acknowledgement}}
  \addto\captionsenglish{\renewcommand{\corollaryname}{Corollary}}
  \addto\captionsenglish{\renewcommand{\definitionname}{Definition}}
  \addto\captionsenglish{\renewcommand{\lemmaname}{Lemma}}
  \addto\captionsenglish{\renewcommand{\propositionname}{Proposition}}
  \addto\captionsenglish{\renewcommand{\remarkname}{Remark}}
  \addto\captionsenglish{\renewcommand{\theoremname}{Theorem}}
  \providecommand{\acknowledgementname}{Acknowledgement}
  \providecommand{\corollaryname}{Corollary}
  \providecommand{\definitionname}{Definition}
  \providecommand{\lemmaname}{Lemma}
  \providecommand{\propositionname}{Proposition}
  \providecommand{\remarkname}{Remark}
\providecommand{\theoremname}{Theorem}

\begin{document}

\title{Multifractal Formalism for Expanding Rational Semigroups and Random
Complex Dynamical Systems}

\author{Johannes Jaerisch and Hiroki Sumi}

\date{July 12, 2015. Published in Nonlinearity {\bf 28} (2015) 2913-2938.}

\thanks{\ \newline \noindent Johannes Jaerisch (corresponding author) \newline  Department of Mathematics,  Faculty of Science and Engineering, 
Shimane University, Nishikawatsu 1060, Matsue, Shimane, 690-8504 Japan \newline E-mail: jaerisch@riko.shimane-u.ac.jp\ \ Web: http://www.math.shimane-u.ac.jp/$\sim$jaerisch/  \newline Tel: +81-(0)8-5232-6377 \newline \ \newline \noindent Hiroki Sumi \newline Department of Mathematics, Graduate School of Science, Osaka University, 1-1 Machikaneyama, Toyonaka, Osaka, 560-0043, Japan\newline E-mail: sumi@math.sci.osaka-u.ac.jp\ \ Web: http://www.math.sci.osaka-u.ac.jp/$\sim$sumi/\newline Tel: +81-(0)6-6850-5307\ \ Fax: +81-(0)6-6850-5327}

\begin{abstract}
We consider the multifractal formalism for the dynamics of semigroups
of rational maps on the Riemann sphere and random complex dynamical
systems. We elaborate a multifractal analysis of level sets given
by quotients of Birkhoff sums with respect to the skew product associated
with a semigroup of rational maps. Applying these results, we perform
a multifractal analysis of the H\"older regularity of limit state
functions of random complex dynamical systems. 
\end{abstract}
\maketitle

\section{Introduction and Statement of Results}

Let $\Rat$ denote the set of all non-constant rational maps on the
Riemann sphere $\Chat$. A subsemigroup of $\Rat$ with semigroup
operation being functional composition is called a \emph{rational
semigroup}. The first study of dynamics of rational semigroups was
conducted by A. Hinkkanen and G. J. Martin (\cite{MR1397693}), who
were interested in the role of polynomial semigroups (i.e., semigroups
of non-constant polynomial maps) while studying various one-complex-dimensional
module spaces for discrete groups, and by F. Ren's group (\cite{MR1435167}),
who studied such semigroups from the perspective of random dynamical
systems. The first study of random complex dynamics was given by J.
E. Fornaess and N. Sibony (\cite{MR1145616}). For the motivations
and the recent studies  on random complex dynamics, see the second
author's work \cite{s11random,S13Coop}. 

The study of multifractals goes back to the work of \cite{Mandelbrot:74,FrischParisi:85,Halsey:86}
which was motivated by statistical physics. In this paper, we perform
a multifractal analysis of level sets given by quotients of Birkhoff
sums with respect to the skew product associated with a rational semigroup.
We recommend \cite{MR1435198,pesindimensiontheoryMR1489237} for a
similar kind of multifractal analysis for conformal repellers. 

One of our main motivations to develop the mutifractal formalism for
rational semigroups is to apply our results to random complex dynamical
systems. The multifractal formalism allows us to investigate level
sets of a prescribed H\"older regularity of limit state functions
(linear combinations of unitary eigenfunctions of transition operators)
of random complex dynamical systems. In this way, our multifractal
analysis exhibits a refined gradation between chaos and order for
random complex dynamical systems, which has been recently studied
by the second author in \cite{s11random,S13Coop}.  We remark that this paper is the first one in which the multifractal formalism is applied to
the study of limit state functions of random complex dynamical systems.
Also, we note that under certain conditions such a limit state function
is continuous on $\Chat$ but varies precisely on the 
Julia set of the associated rational semigroup, which is a thin fractal set.
From this point of view, this study is deeply related to both random
dynamics and fractal geometry. 

Throughout this paper, we will always assume that $I$ is a finite set.
An element $f=\left(f_{i}\right)_{i\in I}\in\left(\Rat\right)^{I}$
is called a \emph{multi-map}. 
We say that \foreignlanguage{english}{$f=\left(f_{i}\right)_{i\in I}$}
is \emph{exceptional} if $\card\left(I\right)=1$ and $\deg\left(f_{i}\right)=1$
for $i\in I$. Throughout, we always assume that $f$ is non-exceptional. 
Let $f=\left(f_{i}\right)_{i\in I}$ be a  multi-map and
let \foreignlanguage{english}{$G=\left\langle f_{i}:i\in I\right\rangle $}, where $\left\langle f_{i}:i\in I\right\rangle$ denotes the rational semigroup generated by   $\left(f_{i}\right)_{i\in I}$, i.e., $G=\{f_{i_1}\circ \dots \circ f_{i_n}\mid  n\in \N, \,\,i_1,\dots ,i_n \in I \}$. Let $\left(p_{i}\right)_{i\in I}$ be a probability vector and let
$\tau:=\sum_{i\in I}p_{i}\delta_{f_{i}}$ where $\delta_{f_{i}}$
refers to the Dirac measure supported on $f_{i}$. We  always assume
that $0<p_{i}<1$ for each $i\in I$.  We consider the i.i.d. random
dynamical system associated with $\tau$, i.e., at every step we choose
a map $f_{i}$ with probability $p_{i}$. We denote by $C\bigl(\Chat\bigr)$
the Banach space of complex-valued continuous functions on $\Chat$
endowed with the supremum norm. The transition operator $M_{\tau}:C\bigl(\Chat\bigr)\rightarrow C\bigl(\Chat\bigr)$
is the bounded linear operator  given by $M_{\tau}\left(\varphi\right)\left(z\right):=\sum_{i\in I}\varphi\left(f_{i}\left(z\right)\right)p_{i}$,  for each $\varphi\in C\bigl(\Chat\bigr)$ and $z\in\Chat$.
A function $\rho\in C\bigl(\Chat\bigr)$ with $\rho \not\equiv 0$ is called a unitary eigenfunction
of\foreignlanguage{english}{ $M_{\tau}$} if there exists $u\in\C$
with $\left|u\right|=1$ such that $M_{\tau}\left(\rho\right)=u\rho$.
Let $U_{\tau}$ be the space of all linear combinations of unitary
eigenfunctions of $M_{\tau}:C\bigl(\Chat\bigr)\rightarrow C\bigl(\Chat\bigr)$. To investigate the regularity of the elements in $U_{\tau}$ we consider the following quantities.
\begin{defn}
\label{def:QRH}Let $\rho:\Chat\rightarrow\C$ be a bounded function and let
$z\in\Chat$. We set 
\[
\M_{*}\left(\rho,z\right):=\liminf_{r\rightarrow0}\frac{\log\M\left(\rho,z,r\right)}{\log r},\;\M^{*}\left(\rho,z\right):=\limsup_{r\rightarrow0}\frac{\log\M\left(\rho,z,r\right)}{\log r}\;\mbox{and}\;\M\left(\rho,z\right):=\lim_{r\rightarrow0}\frac{\log\M\left(\rho,z,r\right)}{\log r},
\]
where $\M\left(\rho,z,r\right)$ is for $r>0$ given by 
\[
\M\left(\rho,z,r\right):=\sup_{y\in B\left(z,r\right)}\left|\rho\left(y\right)-\rho\left(z\right)\right|.
\]
Moreover, we define for each $\alpha\in\R$ the corresponding level
sets
\[
R_{*}\left(\rho,\alpha\right):=\left\{ y\in\Chat:\M_{*}\left(\rho,y\right)=\alpha\right\} ,\; R^{*}\left(\rho,\alpha\right):=\left\{ y\in\Chat:\M^{*}\left(\rho,y\right)=\alpha\right\} 
\]
and 
\[
R\left(\rho,\alpha\right):=\left\{ y\in\Chat:\M\left(\rho,y\right)=\alpha\right\} .
\]
Let $d$ denote the spherical distance on $\Chat$. The pointwise H\"older
exponent \foreignlanguage{english}{$\Hol\left(\rho,z\right)$} of
$\rho$ at $z$ is given by 
\[
\Hol\left(\rho,z\right):=\sup\left\{ \beta\in[0,\infty):\limsup_{y\rightarrow z,y\neq z}\frac{\left|\rho\left(y\right)-\rho\left(z\right)\right|}{d\left(y,z\right)^{\beta}}<\infty\right\} \in\left[0,\infty\right].
\]
The level set $H\left(\rho,\alpha\right)$ with prescribed H\"older
exponent $\alpha\in\R$ is given by 
\[
H\left(\rho,\alpha\right):=\left\{ y\in\Chat:\Hol\left(\rho,y\right)=\alpha\right\} .
\]
\end{defn}
In fact, we will show in Lemma \ref{lem:hoelderexponent-via-Q} that $\Hol(\rho,z)=Q_{*}(\rho,z)$ for every $z\in \Chat$. We refer to Section \ref{sec:Application-to-Random} for the proof of this fact and for further properties of the quantities introduced in Definition \ref{def:QRH}.

We proceed by introducing the necessary  preliminaries
to state our first main result. The \emph{Fatou set} $F\left(G\right)$ and the \emph{Julia set} $J\left(G\right)$
of a rational semigroup $G$ are given by 
\[
F\left(G\right):=\left\{ z\in\Chat:G\mbox{ is normal in a neighbourhood of }z\right\} \quad\mbox{ and }\quad J\left(G\right):=\Chat\setminus F\left(G\right).
\]
If  $G$ is a rational semigroup generated by   a single map $g\in\Rat$, then we write $G=\left\langle g\right\rangle $. Moreover, for a single map $g\in\Rat$, we set $F\left(g\right):=F\left(\left\langle g\right\rangle \right)$
and $J\left(g\right):=J\left(\left\langle g\right\rangle \right)$. 

Let $f=\left(f_{i}\right)_{i\in I}$ be a  multi-map. The skew product associated with the  multi-map
$f=\left(f_{i}\right)_{i\in I}$ is given by 
\[
\tilde{f}:I^{\N}\times\hat{\C}\rightarrow I^{\N}\times\hat{\C},\quad\tilde{f}\left(\omega,z\right):=\left(\sigma\left(\omega\right),f_{\omega_{1}}\left(z\right)\right),
\]
where $\sigma:I^{\N}\rightarrow I^{\N}$ denotes the left-shift map
given by $\sigma\left(\omega_{1},\omega_{2},\dots\right):=\left(\omega_{2},\omega_{3},\dots\right)$,
for $\omega=\left(\omega_{1},\omega_{2},\dots\right)\in I^{\N}$.
We say that a multi-map $f=\left(f_{i}\right)_{i\in I}$ is \emph{expanding}
if the associated skew product $\tilde{f}$ is expanding along fibres
on the Julia set $J\left(\tilde{f}\right)$ (see Definition \ref{def:expanding}). 

We say that $\psi=\left(\psi_{i}\right)_{i\in I}$
is a \emph{H\"older family associated with  the multi-map} $f=\left(f_{i}\right)_{i\in I}$
if $\psi_{i}:f_{i}^{-1}\left(J\left(G\right)\right)\rightarrow\R$
is H\"older continuous for each $i\in I$, where $G=\left\langle f_{i}:i\in I\right\rangle $
and $J\left(G\right)$ is equipped with the metric inherited from
the spherical distance $d$ on $\Chat.$ 
Note that $\cup _{i\in I}f_{i}^{-1}(J(G))=J(G)$ (\cite[Lemma 2.4]{MR1767945}). 
If it is clear from the context
with which multi-map $\psi$ is associated, then we simply say that
$\psi$ is a H\"older family. For a H\"older family $\psi=\left(\psi_{i}\right)_{i\in I}$,
we define $\tilde{\psi}:J\left(\tilde{f}\right)\rightarrow\R$ given
by $\tilde{\psi}\left(\omega,z\right):=\psi_{\omega_{1}}\left(z\right)$,
for all $\omega=\left(\omega_{1},\omega_{2},\dots\right)\in I^{\N}$
and $z\in f_{\omega_{1}}^{-1}\left(J\left(G\right)\right)$, and for
each $n\in\N$ we denote by $S_{n}\tilde{\psi}:J\left(\tilde{f}\right)\rightarrow\R$
the Birkhoff sum of $\tilde{\psi}$ with respect to $\tilde{f}$ given
by $S_{n}\tilde{\psi}:=\sum_{i=0}^{n-1}\tilde{\psi}\circ\tilde{f}^{i}$. 

For an expanding multi-map $f=\left(f_{i}\right)_{i\in I}$, let $\zeta=\left(\zeta_{i}:f_{i}^{-1}\left(J\left(G\right)\right)\rightarrow\R\right)_{i\in I}$
be the H\"older family given by $\zeta_{i}\left(z\right):=-\log\left\Vert f_{i}'\left(z\right)\right\Vert $
for each $i\in I$ and $z\in f_{i}^{-1}\left(J\left(G\right)\right)$,
where $\Vert\cdot\Vert$ denotes the norm of the derivative with respect
to the spherical metric on $\Chat$. Let $\pi_{\Chat}:I^{\N}\times\Chat\rightarrow\Chat$
denote the canonical projection. We  define the level sets $\mathcal{F}\left(\alpha,\psi\right)$,
 which are for $\alpha\in\R$ given by
\[
\mathcal{F}\left(\alpha,\psi\right):=\pi_{\Chat}\left(\tilde{\mathcal{F}}\left(\alpha,\psi\right)\right),\mbox{ where }\tilde{\mathcal{F}}\left(\alpha,\psi\right):=\biggl\{ x\in J\left(\tilde{f}\right):\lim_{n\rightarrow\infty}\frac{S_{n}\tilde{\psi}\left(x\right)}{S_{n}\tilde{\zeta}\left(x\right)}=\alpha\biggr\}.
\]
The \emph{(Hausdorff-) dimension spectrum $l$ of $\left(f,\psi\right)$}
is given by 
\[
l\left(\alpha\right):=\dim_{H}\left(\mathcal{F}\left(\alpha,\psi\right)\right),\mbox{ for }\alpha\in\R.
\]
The range of the multifractal spectrum is given by 
\[
\alpha_{-}\left(\psi\right):=\inf\left\{ \alpha\in\R:\mathcal{F}\left(\alpha,\psi\right)\neq\emptyset\right\} \quad\mbox{and }\quad\alpha_{+}\left(\psi\right):=\sup\left\{ \alpha\in\R:\mathcal{F}\left(\alpha,\psi\right)\neq\emptyset\right\} .
\]

The \emph{free energy function for $\left(f,\psi\right)$} is the
unique function $t:\R\rightarrow\R$ such that $\mathcal{P}\bigl(\beta\tilde{\psi}+t\left(\beta\right)\tilde{\zeta},\tilde{f}\bigr)=0$
for each $\beta\in\R$, where $\mathcal{P}\left(\cdot,\tilde{f}\right)$
denotes the topological pressure with respect to $\tilde{f}$ (\cite{MR648108}).
The number $t\left(0\right)$ is also referred to as the \emph{critical
exponent} $\delta$ of $f$. The convex conjugate of $t$ (\cite[Section 12]{rockafellar-convexanalysisMR0274683})
is given by 
\[
t^{*}:\R\rightarrow\R\cup\left\{ \infty\right\} ,\quad t^{*}\left(c\right):=\sup_{\beta\in\R}\left\{ \beta c-t\left(\beta\right)\right\} ,\quad c\in\R.
\]
We say that $f=\left(f_{i}\right)_{i\in I}$ satisfies the \emph{separation condition}
if $f_{i}^{-1}\left(J\left(G\right)\right)\cap f_{j}^{-1}\left(J\left(G\right)\right)=\emptyset$
for all $i,j\in I$ with $i\neq j$, where $G:=\left\langle f_{i}:i\in I\right\rangle $.
Note that in this case, under the assumption $J(G)\neq \emptyset $, for any probability vector $\left(p_{i}\right)_{i\in I}\in\left(0,1\right)^{I}$,
setting $\tau:=\sum_{i\in I}p_{i}\delta_{f_{i}}$, we have that (1)
$1\le\dim_{\C}\left(U_{\tau}\right)<\infty$ and (2) there exists
a bounded linear operator $\pi_{\tau}:C\bigl(\Chat\bigr)\rightarrow U_{\tau}$
such that, for each $\varphi\in C\bigl(\Chat\bigr)$, we have $\Vert M_{\tau}^{n}\left(\varphi-\pi_{\tau}\left(\varphi\right)\right)\Vert\rightarrow0$,
as $n$ tends to infinity (see \cite[Theorem 3.15]{s11random}). If
an element $\rho\in U_{\tau}$ is non-constant, then $\rho$ is continuous
on $\Chat$ and the set of varying points of $\rho$ is equal to the
thin fractal set $J(G)$ (for the figure of the graph of such a function
$\rho$, see \cite{s11random}). Such $\rho$ is considered as a complex
analogue of a devil's staircase or Lebesgue's singular function. Some
of such functions $\rho$ are called devil's coliseums (see \cite{s11random}). Our first main result shows that, for every $\rho \in U_{\tau}$ non-constant,  the level sets $R_{*}$, \foreignlanguage{english}{$R$,
$R^{*}$ and} $H$ satisfy the multifractal formalism.
\begin{thm}
[Theorem \ref{thm:mf-for-hoelderexponent}]  \label{thm:-regularity-intro}Let
$f=\left(f_{i}\right)_{i\in I}$ be an expanding multi-map and let
\foreignlanguage{english}{\textup{$G=\left\langle f_{i}:i\in I\right\rangle $.}}
Suppose that $f$ satisfies the separation condition. Let $\left(p_{i}\right)_{i\in I}\in\left(0,1\right)^{I}$
be a probability vector and let $\tau:=\sum_{i\in I}p_{i}\delta_{f_{i}}$.
Suppose there exists a non-constant function belonging to $U_{\tau}$.
Let $\rho\in C\bigl(\Chat\bigr)$ be a non-constant function belonging
to $U_{\tau}$. Let $\psi=\left(\psi_{i}:f_{i}^{-1}\left(J\left(G\right)\right)\rightarrow\R\right)_{i\in I}$
be given by $\psi_{i}\left(z\right):=\log p_{i}$. Let $t:\R\rightarrow\R$
denote the free energy function for  $\left(f,\psi\right)$. Then
we have the following.
\begin{enumerate}
\item There exists a number $a\in\left(0,1\right)$ such that $\rho:\Chat\rightarrow\C$
is $a$-H\"older continuous and $a\le\alpha_{-}\left(\psi\right)$.
\selectlanguage{english}%
\item \textup{\emph{We have $\alpha_{+}\left(\psi\right)=\sup\left\{ \alpha\in\R:H\left(\rho,\alpha\right)\neq\emptyset\right\} $
and $\alpha_{-}\left(\psi\right)=\inf\left\{ \alpha\in\R:H\left(\rho,\alpha\right)\neq\emptyset\right\} $.
Moreover, $H$ can be replaced by $R_{*},R$ or $R^{*}$. }}
\selectlanguage{british}%
\item Let $\alpha_{\pm}:=\alpha_{\pm}\left(\psi\right)$. If $\alpha_{-}<\alpha_{+}$
then we have for each $\alpha\in\left(\alpha_{-},\alpha_{+}\right)$,
\[
\dim_{H}\left(R^{*}\left(\rho,\alpha\right)\right)=\dim_{H}\left(R_{*}\left(\rho,\alpha\right)\right)=\dim_{H}\left(R\left(\rho,\alpha\right)\right)=\dim_{H}\left(H\left(\rho,\alpha\right)\right)=-t^{*}\left(-\alpha\right)>0.
\]
 Moreover, $s(\alpha ):=-t^{*}\left(-\alpha\right)$ defines a real
analytic and strictly concave positive function on $\left(\alpha_{-},\alpha_{+}\right)$
with maximum value $\delta>0$. Also, $s''<0$ on $(\alpha _{-}, \alpha _{+}).$ 
\item 

\begin{enumerate}
\item For each $i\in I$ we have $\deg\left(f_{i}\right)\ge2$. Moreover,
we have $\alpha_{-}=\alpha_{+}$ if and only if there exists an automorphism
$\varphi\in\Aut\bigl(\Chat\bigr)$, complex numbers $\left(a_{i}\right)_{i\in I}$
and $\lambda\in\R$ such that for all $i\in I$, 
\[
\varphi\circ f_{i}\circ\varphi^{-1}\left(z\right)=a_{i}z^{\pm\deg\left(f_{i}\right)}\quad\mbox{and}\quad\log\deg\left(f_{i}\right)=\lambda\log p_{i}.
\]

\item If $\alpha_{-}=\alpha_{+}$ then we have
\[
R^{*}\left(\rho,\alpha_{-}\right)=R_{*}\left(\rho,\alpha_{-}\right)=R\left(\rho,\alpha_{-}\right)=H\left(\rho,\alpha_{-}\right)
=J(G),
\]
where $\dim_H(J(G))=\delta>0$ and  $R^{*}\left(\rho,\alpha\right)=R_{*}\left(\rho,\alpha\right)=R\left(\rho,\alpha\right)=H\left(\rho,\alpha\right)=\emptyset$
for all $\alpha\neq\alpha_{-}$.
\end{enumerate}
\end{enumerate}
\end{thm}
We denote by $C^{\alpha}(\Chat):=\{\varphi:\Chat \rightarrow \C \mid \Vert\varphi\Vert_{\alpha}<\infty \}$ the Banach space of  $\alpha$-H\"older continuous
functions on $\Chat$ endowed with the $\alpha$-H\"older norm 
\[
\Vert\varphi\Vert_{\alpha}:=\sup_{z\in\Chat}\left|\varphi\left(z\right)\right|+\sup_{x,y\in\Chat,x\neq y}\left|\varphi\left(x\right)-\varphi\left(y\right)\right|/d(x,y)^{\alpha}.
\]
Under the assumptions of Theorem \ref{thm:-regularity-intro} and
some additional conditions, we have $\alpha_{-}<1$ (see the Remark
in Section \ref{sec:Examples}). In this case Theorem \ref{thm:-regularity-intro}  implies
that for each $\alpha\in\left(\alpha_{-},1\right)$ the iteration
of the transition operator $M_{\tau}$ does not behave well on the
Banach space $C^{\alpha}(\Chat)$, 
i.e., there exists an element $\varphi\in C^{\alpha}(\Chat)$ such
that $\Vert M_{\tau}^{n}(\varphi)\Vert_{\alpha}\rightarrow\infty$
as $n\rightarrow\infty$. It means that, regarding the iteration of the transition operator $M_{\tau}$ on functions spaces,  even though the chaos disappears
on $C(\Chat)$ and $C^{a}(\Chat)$, we still have a kind of complexity
(or chaos) on the space $C^{\alpha}(\Chat)$ for each $\alpha\in(\alpha_{-},1)$.
Thus, in random complex dynamical systems we sometimes have a kind
of gradation between chaos and order. Theorem \ref{thm:-regularity-intro}
can be seen as a refinement of this gradation. In Section \ref{sec:Examples}
we give many examples to which we can apply Theorem \ref{thm:-regularity-intro}.

In order to show our first main result, we prove the general multifractal formalism for level sets given
by quotients of Birkhoff sums with respect to the skew product associated
with a semigroup of rational maps.   We say that a multi-map  $f=\left(f_{i}\right)_{i\in I}$ satisfies the \emph{open
set condition} if there exists a non-empty open set $U$ in $\Chat$
such that $\left\{ f_{i}^{-1}\left(U\right):i\in I\right\} $ consists
of pairwise disjoint subsets of $U$. Note that the open set condition is weaker than the separation condition, since the higher iterates of an expanding multi-map satisfying the separation condition also satisfy the open set condition.   

Our second  main result shows that, for an expanding multi-map  satisfying
only the open set condition and for an arbitrary H\"older family $\psi$,
the level sets $\mathcal{F}\left(\alpha,\psi\right)$ satisfy the
multifractal formalism. 
\begin{thm}
[Theorem \ref{thm:multifractalformalism}]  \label{thm:-mf-intro}Let
$f=\left(f_{i}\right)_{i\in I}$ be an expanding multi-map which satisfies
the open set condition. Let $\psi=\left(\psi_{i}\right)_{i\in I}$
be a H\"older family associated with $f$ and let $t:\R\rightarrow\R$
denote the free energy function for  $\left(f,\psi\right)$. Suppose
that there exists $\gamma\in\R$ such that $\mathcal{P}\bigl(\gamma\tilde{\psi},\tilde{f}\bigr)=0$
and suppose that $\alpha_{-}\left(\psi\right)<\alpha_{+}\left(\psi\right)$.
Then (1) the Hausdorff dimension spectrum \emph{$l$ of $\left(f,\psi\right)$}
is real analytic and strictly concave on $\left(\alpha_{-}\left(\psi\right),\alpha_{+}\left(\psi\right)\right)$
with maximal value $\delta$, (2) $l''<0$ on $(\alpha _{-}(\psi ),\alpha _{+}(\psi ))$, and (3) for $\alpha\in\left(\alpha_{-}\left(\psi\right),\alpha_{+}\left(\psi\right)\right)$
we have that 
\[
l\left(\alpha\right)=-t^{*}\left(-\alpha\right)>0.
\]
\end{thm}
\begin{rem*}
Note that in Theorem \ref{thm:-mf-intro} we only assume that the multi-map $f$ satisfies the open set condition, which does not imply that $f$ satisfies the  separation condition. In fact, there are many $2$-generator expanding polynomial semigroups satisfying the open set condition for which the Julia set is connected (see \cite{twogenerator}).  
Theorem \ref{thm:-mf-intro} can be applied even to such semigroups. 
Moreover, let us remark that, since each map of the generator system $\left(f_{i}\right)_{i\in I}$
is not injective in general, we need much more efforts in the proof
of Theorem \ref{thm:-mf-intro} than in the case of contracting iterated
function systems, even if the open set condition is assumed.
\end{rem*}

\begin{rem*}
Note that, in general, we cannot replace the Hausdorff dimension by the box-counting dimension in Theorem \ref{thm:-mf-intro}. In fact,
if $\alpha=-t'\left(\beta\right)$ for some $\beta\in\R$, then we
have that $\nu_{\beta}\left(\mathcal{F}\left(\alpha,\psi\right)\right)=1$
and $\supp\left(\nu_{\beta}\right)=J\left(G\right)$ by Lemmas \ref{lem:support-of-nu-beta}
and \ref{lem:existence-sub-conformalmeasure-on-juliaset}, where $\nu_{\beta}$
is given in Definition \ref{def-betaconformal-eigenmeasures}. Hence,
$\mathcal{F}\left(\alpha,\psi\right)$ is dense in $J\left(G\right)$,
which implies that $\dim_{B}\left(\mathcal{F}\left(\alpha,\psi\right)\right)=\dim_{B}\left(J\left(G\right)\right)$,
where $\dim_{B}$ refers to the box-counting dimension. 
\end{rem*}
The results stated in Proposition \ref{prop:intro-degenerate-criterion}
and Corollary \ref{cor:intro-lyapunovdegeneracy} below follow from
the general theory without assuming the open set condition (see Theorem
\ref{thm:multifractalformalism} (\ref{enu:range-of-spectrum}) and
(\ref{enu:upperbound-of-spectrum})). If each potential $\psi_{i}$
is constant, then we have the following criterion for a non-trivial
multifractal spectrum. 
\begin{prop}
[Proposition  \ref{lem:degenerate-via-zdunik}]  \label{prop:intro-degenerate-criterion}Let
$f=\left(f_{i}\right)_{i\in I}$ be an expanding multi-map and let
\foreignlanguage{english}{\textup{$G=\left\langle f_{i}:i\in I\right\rangle $.}}
(We do not assume the open set condition.) Suppose that $\deg\left(f_{i_{0}}\right)\ge2$
for some $i_{0}\in I$. Let $\left(c_{i}\right)_{i\in I}$ be a sequence
of negative numbers and let the H\"older family $\psi=\left(\psi_{i}:f_{i}^{-1}\left(J\left(G\right)\right)\rightarrow\R\right)_{i\in I}$
be given by $\psi_{i}\left(z\right)=c_{i}$ for each $z\in f_{i}^{-1}\left(J\left(G\right)\right)$.
 Then we have $\alpha_{-}\left(\psi\right)=\alpha_{+}\left(\psi\right)$
if and only if there exist an automorphism $\varphi\in\Aut\bigl(\Chat\bigr)$,
complex numbers $\left(a_{i}\right)_{i\in I}$ and $\lambda\in\R$
such that for all $i\in I$, 
\[
\varphi\circ f_{i}\circ\varphi^{-1}\left(z\right)=a_{i}z^{\pm\deg\left(f_{i}\right)}\quad\mbox{and}\quad\log\deg\left(f_{i}\right)=\lambda c_{i}.
\]

\end{prop}

\begin{rem*}
Let us point out the relation between Proposition \ref{prop:intro-degenerate-criterion} and rigidity results in thermodynamic formalism. There are several equivalent formulations to characterize when the multifractal spectrum degenerates. Namely,  the equality  
$\alpha_{-}(\psi )=\alpha_{+}(\psi )$ is equivalent to each of the following statements: (1) The equilibrium states $\tilde{\mu }_{0}$ and $\tilde{\mu }_{\gamma }$  coincide. (2) The graph of the free energy function $t$ is a straight line.  
(3) There exists a 
continuous function $u:J(\tilde{f})\rightarrow \Bbb{R}
 $ such that 
$\delta \tilde{\zeta }-\gamma \tilde{\psi }=u-u\circ \tilde{f}.$  
For precise statement and the proof, we refer to 
Proposition \ref{prop:characterisation-strict-convexity} and the proofs of 
Theorem~\ref{thm:multifractalformalism} and Proposition~\ref{lem:degenerate-via-zdunik}.  
\end{rem*}

\begin{rem*} 
The proof of Proposition \ref{prop:intro-degenerate-criterion} makes use of a rigidity result of Zdunik (\cite{MR1032883}) for the classical iteration of  a single rational map. We give an extension of this result to rational semigroups. We emphasize that the map $\varphi$ in Proposition \ref{prop:intro-degenerate-criterion}  is independent of $i\in I$.
\end{rem*}

An interesting special case is given by the Lyapunov spectrum of an
expanding multi-map $f=\left(f_{i}\right)_{i\in I}$. The Lyapunov
spectrum is given by the level sets $\mathcal{L}\left(\alpha\right)$,
$\alpha\in\R$, where we define 
\[
\mathcal{L}\left(\alpha\right):=\mathcal{\mathcal{F}}\left(\alpha,-\1\right)=\left\{ z\in\Chat:\exists\omega\in I^{\N}\mbox{ such that }(\omega, z)\in J(\tilde{f}) \mbox{ and }\lim_{n\rightarrow\infty}\frac{n}{\log\Vert\left(f_{\omega_{n}}\circ\dots\circ f_{\omega_{1}}\right)'\left(z\right)\Vert}=\alpha\right\} .
\]
We say that $f=\left(f_{i}\right)_{i\in I}$ has trivial Lyapunov
spectrum if there exists $\alpha_{0}\in\R$ such that $\mathcal{L}\left(\alpha\right)=\emptyset$
if $\alpha\neq\alpha_{0}$. 
\begin{cor}
[Lyapunov spectrum] \label{cor:intro-lyapunovdegeneracy}Let $f=\left(f_{i}\right)_{i\in I}$
be an expanding multi-map. Suppose that $\deg\left(f_{i_{0}}\right)\ge2$
for some $i_{0}\in I$. Then $f$ has trivial Lyapunov spectrum if
and only if there exist an automorphism $\varphi\in\Aut\bigl(\Chat\bigr)$
and complex numbers $\left(a_{i}\right)_{i\in I}$ such that $\varphi\circ f_{i}\circ\varphi^{-1}\left(z\right)=a_{i}z^{\pm\deg\left(f_{i_{0}}\right)}$. 
\end{cor}

The paper is organised as follows. In Section \ref{sec:Preliminaries}
we collect the necessary preliminaries on the dynamics of expanding
rational semigroups. In Section \ref{sec:Thermodynamic-formalism-for}
we recall basic facts from thermodynamic formalism for expanding dynamical
systems in the framework of the skew product associated with an expanding
rational semigroup. In Section \ref{sec:Multifractal-Formalism} we
investigate the local dimension of conformal measures supported on
subsets of the Julia set of rational semigroups satisfying the open
set condition and we prove a multifractal formalism for H\"older
continuous potentials in Theorem \ref{thm:multifractalformalism}.
In Section \ref{sec:Application-to-Random}  we apply the multifractal
formalism to investigate the H\"older regularity of linear combinations
of unitary eigenfunctions of transition operators in random complex
dynamics. Finally, examples of our results are given in Section \ref{sec:Examples}. 
\begin{acknowledgement*}
The authors would like to thank Rich Stankewitz for valuable comments. The authors would like to thank the referee for a careful reading
of the manuscript and for the helpful comments. 
The research of the first author was supported by the research fellowship
JA 2145/1-1 of the German Research Foundation (DFG) and the JSPS Postdoctoral Fellowship (ID No. P14321). The research
of the second author was partially supported by JSPS KAKENHI 24540211. 
\end{acknowledgement*}

\section{Preliminaries \label{sec:Preliminaries}}

Let $f=\left(f_{i}\right)_{i\in I}\in\left(\Rat\right)^{I}$ be a 
multi-map and let \foreignlanguage{english}{$G=\left\langle f_{i}:i\in I\right\rangle $.
For $n\in\N$ and $\left(\omega_{1},\omega_{2},\dots,\omega_{n}\right)\in I^{n}$,
we set $f_{\left(\omega_{1},\omega_{2},\dots,\omega_{n}\right)}:=f_{\omega_{n}}\circ f_{\omega_{n-1}}\circ\dots\circ f_{\omega_{1}}$.
For $\omega\in I^{\N}$ we set $\omega_{|n}:=\left(\omega_{1},\omega_{2},\dots,\omega_{n}\right)$
and we define} 
\[
F_{\omega}:=\left\{ z\in\Chat:\Bigl(f_{\omega_{|n}}\Bigr)_{n\in\N}\mbox{ is normal in a neighbourhood of }z\right\} \mbox{ and }J_{\omega}:=\Chat\setminus F_{\omega}.
\]
The \emph{skew product }associated with\emph{ $f=\left(f_{i}\right)_{i\in I}$}
is given by 
\[
\tilde{f}:I^{\N}\times\hat{\C}\rightarrow I^{\N}\times\hat{\C},\quad\tilde{f}\left(\omega,z\right):=\left(\sigma\left(\omega\right),f_{\omega_{1}}\left(z\right)\right),
\]
where $\sigma:I^{\N}\rightarrow I^{\N}$ denotes the left-shift map
given by $\sigma\left(\omega_{1},\omega_{2},\dots\right):=\left(\omega_{2},\omega_{3},\dots\right)$,
for $\omega=\left(\omega_{1},\omega_{2},\dots\right)\in I^{\N}$. 

For each $\omega\in I^{\N}$, we set $J^{\omega}:=\left\{ \omega\right\} \times J_{\omega}$
and we set 
\[
J\left(\tilde{f}\right):=\overline{\bigcup_{\omega\in I^{\N}}J^{\omega}},\quad F\left(\tilde{f}\right):=\left(I^{\N}\times\hat{\C}\right)\setminus J\left(\tilde{f}\right),
\]
where the closure is taken with respect to the product topology on
$I^{\N}\times\Chat$. Let $\pi_{1}:I^{\N}\times\Chat\rightarrow I^{\N}$
and $\pi_{\Chat}:I^{\N}\times\Chat\rightarrow\Chat$ denote the canonical
projections. We refer to \cite[Proposition 3.2]{MR1767945} for the
proof of the following proposition. 
\begin{prop}
\label{prop:basic-facts}Let $f=\left(f_{i}\right)_{i\in I}\in\left(\Rat\right)^{I}$
and let $G=\left\langle f_{i}:i\in I\right\rangle $.  Let $\tilde{f}:I^{\N}\times\hat{\C}\rightarrow I^{\N}\times\hat{\C}$ be the skew product associated with $f$.  Then we have the following. \end{prop}
\begin{enumerate}
\item \label{enu:basic-facts-1}$\tilde{f}\left(J^{\omega}\right)=J^{\sigma\omega}$ and $(\tilde{f}_{|\pi_{1}^{-1}\left(\omega\right)})^{-1}\left(J^{\sigma\omega}\right)=J^{\omega}$, for each $\omega\in I^{\N}$. 
\item \label{enu:basic-facts-2}$\tilde{f}\left(J\left(\tilde{f}\right)\right)=J\left(\tilde{f}\right),$ $\tilde{f}^{-1}\left(J\left(\tilde{f}\right)\right)=J\left(\tilde{f}\right)$, $\tilde{f}\left(F\left(\tilde{f}\right)\right)=F\left(\tilde{f}\right),$ $\tilde{f}^{-1}\left(F\left(\tilde{f}\right)\right)=F\left(\tilde{f}\right)$. 
\item \label{enu:basic-facts-3}Let $G=\left\langle f_{i}:i\in I\right\rangle $ and suppose that $\card\left(J\left(G\right)\right)\ge3$. Then we have $J\left(\tilde{f}\right)=\bigcap_{n\in\N_{0}}\tilde{f}^{-n}\left(I^{\N}\times J\left(G\right)\right)$ and $\pi_{\Chat}\left(J\left(\tilde{f}\right)\right)=J\left(G\right)$. Here, $\N_0 := \N \cup \{ 0\}$.\end{enumerate}
\begin{defn}
Let $G$ be a rational semigroup and let  $z\in\Chat$. The backward orbit $G^{-}\left(z\right)$ of $z$ and the set of exceptional points $E\left(G\right)$ are defined by $G^{-}\left(z\right):=\bigcup_{g\in G}g^{-1}\left(z\right)$ and $E\left(G\right):=\left\{ z\in\Chat:\card\left(G^{-}\left(z\right)\right)<\infty\right\} $.  We say that  a set $A\subset \Chat$ is $G$-backward invariant, if $g^{-1}(A)\subset A$, for each $g\in G$, and we say that  $A$ is  $G$-forward  invariant, if $g(A)\subset A$, for each $g\in G$. 
\end{defn}
The following is proved in \cite{MR1397693} (see also \cite[Lemma 2.3]{MR1767945}, \cite{ MR2900562}). 
\begin{lem}
\label{lem:ratsemi-facts-basic}The following holds for a rational semigroup $G$. \renewcommand{\theenumi}{\alph{enumi}} \begin{enumerate} \item $F(G)$ is $G$-forward invariant and $J(G)$ is $G$-backward invariant.  \item \label{enu:perfectset}If $\card\left(J\left(G\right)\right)\ge3$, then $J\left(G\right)$ is a perfect set. \item \label{enu:atmost-two-exceptionalpoints}If $\card\left(J\left(G\right)\right)\ge3$, then $\card\left(E\left(G\right)\right)\le2$.  \item \label{enu:preimages-dense}If $z\in\Chat\setminus E\left(G\right)$, then $J(G)\subset \overline{G^{-}\left(z\right)}.$ In particular, if $z\in J\left(G\right)\setminus E\left(G\right)$ then $\overline{G^{-}\left(z\right)}=J\left(G\right)$.  \item \label{enu:julia-is-smallestbackwardinvariant}If $\card\left(J\left(G\right)\right)\ge3$, then $J\left(G\right)$ is the smallest closed set containing at least three points which is $G$-backward invariant. \item \label{enu:densityofrepellingfixedpoints}If $\card\left(J\left(G\right)\right)\ge3$, then  \[ J\left(G\right)=\overline{\left\{ z\in\Chat:z\text{ is a repelling fixed point of some }g\in G\right\} }=\overline{\bigcup_{g\in G}J\left(g\right)}, \] where the closure is taken in $\Chat$.  \end{enumerate}
\end{lem}
For a holomorphic map $h:\Chat\rightarrow\Chat$ and $z\in\Chat$, the norm of the derivative of $h$ at $z\in\Chat$ with respect to the spherical metric is denoted by $\left\Vert h'\left(z\right)\right\Vert $. 
\begin{defn}
\label{def:expanding}Let $f=\left(f_{i}\right)_{i\in I}\in\left(\Rat\right)^{I}$
and and let $\tilde{f}:I^{\N}\times\Chat\rightarrow I^{\N}\times\Chat$ denote the associated skew product.  For each $n\in\N$ and $\left(\omega,z\right)\in J\left(\tilde{f}\right)$, we set $\left(\tilde{f}^{n}\right)'\left(\omega,z\right):=(f_{\omega |_{n}})'(z)$.  We say that  $\tilde{f}$ is expanding along fibers if $J\left(\tilde{f}\right)\neq\emptyset$ and if there exist constants $C>0$ and $\lambda>1$ such that for all $n\in\N$,  \[ \inf_{\left(\omega,z\right)\in J\left(\tilde{f}\right)}\Vert\left(\tilde{f}^{n}\right)'\left(\omega,z\right)\Vert\ge C\lambda^{n}, \] where $\Vert\left(\tilde{f}^{n}\right)'\left(\omega,z\right)\Vert$ denotes the norm of the derivative of $f_{\omega_{n}}\circ f_{\omega_{n-1}}\circ\dots\circ f_{\omega_{1}}$ at $z$ with respect to the spherical metric. $f=(f_i)_{i\in I}$ is called expanding if $\tilde{f}$ is expanding along fibers. Also, $G=\langle f_i :i\in I \rangle$ is called expanding if $f=(f_i)_{i\in I}$ is expanding.\end{defn}
\begin{rem}
It follows from Proposition \ref{prop:expandingness-criteria} below
that, for a rational semigroup $G=\left\langle f_{i}:i\in I\right\rangle $,
the notion of expandingness is independent of the choice of the generator
system.
\end{rem}
The following lemma was proved in \cite[Theorem 2.14]{MR1827119}
(see also Proposition \ref{prop:expandingness-criteria} below).
\begin{lem}
\label{lem:exp-implies-compactfibres}If \foreignlanguage{english}{\textup{$f=\left\langle f_{i}:i\in I\right\rangle $
is expanding, then we have $J\left(\tilde{f}\right)=\bigcup_{\omega\in I^{\N}}J^{\omega}$. }}\end{lem}
\begin{defn}
\label{def:hyperbolic}A rational semigroup $G$ is hyperbolic if $P\left(G\right)\subset F\left(G\right)$, where $P\left(G\right)$ denotes the postcritical set of $G$ given by  \[ P\left(G\right):=\overline{\bigcup_{g\in G}\CV\left(g\right)} \] and where $\CV\left(g\right)$ denotes the set of critical values of $g$.\end{defn}
\begin{rem*}
Let $G=\langle f_i :i\in I \rangle$.  Since    $P\left(G\right)=\overline{\bigcup_{g\in G\cup \{ \id \} }g\left(\bigcup_{i\in I}\CV\left(f_{i}\right)\right)}$, we have that   $P(G)$ is $G$-forward invariant. 
\end{rem*}
In the next proposition we give necessary and sufficient conditions
for a rational semigroup to be expanding. We refer to \cite{MR1625124}
for the proofs. For $g\in \Aut(\Chat)$, we say that $g$ is \emph{loxodromic}
if $g$ has exactly two fixed points, for which the modulus of the
multipliers is not equal to one.
\begin{prop}
\label{prop:expandingness-criteria}Let $f=\left(f_{i}\right)_{i\in I}\in\left(\Rat\right)^{I}$
and \foreignlanguage{english}{\textup{$G=\left\langle f_{i}:i\in I\right\rangle $. }}
\begin{enumerate}
\item \label{enu:exp-implies-hyperbolicloxodromic}If $f$ is expanding,
then $G$ is hyperbolic, each element $g\in G$ with $\deg\left(g\right)=1$
is loxodromic, and there exists a $G$-forward invariant non-empty
compact subset of $F\left(G\right)$. 
\item If there exists $g\in G$ with $\deg\left(g\right)\ge2$, if $G$
is hyperbolic and if each element $g\in G$ with $\deg\left(g\right)=1$
is loxodromic, then $f$ is expanding.
\item If $G\subset\Aut(\Chat)$ and each element $g\in G$ is loxodromic,
and if there exists a $G$-forward invariant non-empty compact subset
of $F\left(G\right)$, then $f$ is expanding. 
\end{enumerate}
\end{prop}
Finally, we state the following facts about the exceptional set of
an expanding rational semigroup.
\begin{lem}
\label{lem:moebiussemigroup-expanding-implies-goodexceptionalset}
Let \foreignlanguage{english}{\textup{$G=\left\langle f_{i}:i\in I\right\rangle $}}
denote an expanding rational semigroup. Suppose that $\card\left(J\left(G\right)\right)\ge3$.
Then  we have $E\left(G\right)\subset F\left(G\right)$. \end{lem}
\begin{proof}
Suppose for a contradiction that there exists $z_{0}\in E\left(G\right)\cap J\left(G\right)$. Since $\card\left(J\left(G\right)\right)\ge3$, it follows from the density of the repelling fixed points (Lemma \ref{lem:ratsemi-facts-basic} (\ref{enu:densityofrepellingfixedpoints})) that there exist $z_{1}\in J\left(G\right)$ and $g_{1}\in G$, such that $z_{1}\neq z_{0}$, $g_{1}\left(z_{1}\right)=z_{1}$ and $\Vert g_{1}'\left(z_{1}\right)\Vert>1$. Furthermore, we have  $\card\left(E\left(G\right)\right)\le2$ by Lemma \ref{lem:ratsemi-facts-basic} (\ref{enu:atmost-two-exceptionalpoints}). Combining with the fact that $g^{-1}\left(E\left(G\right)\right)\subset E\left(G\right)$  for each $g\in G$, we conclude that $g_{1}^{2}\left(z_{0}\right)=z_{0}$. Since $G$ is expanding, we have that either $\deg(g_1)\ge 2$ or that  $g_1$ is a  loxodromic M\"obius transformation  by Proposition  \ref{prop:expandingness-criteria} (\ref{enu:exp-implies-hyperbolicloxodromic}). Thus,  we have that  $z_{0}$ is an  attracting fixed point of $g_{1}^{2}$. Let $V$ be a neighborhood of $z_0$ and let $0<c<1$ such that $g_1^2(V)\subset V$ and $\Vert (g_1^2)'(z)\Vert <c$, for each $z\in V$.  By Lemma \ref{lem:ratsemi-facts-basic} (\ref{enu:perfectset})  there exists a sequence $(a_n)$ with $a_n \in J(G)\setminus \{ z_0\}$  such that $\lim_n a_n =z_0$.  Then there exists a sequence $(n_k)\in \N^{\N}$ tending to infinity and a sequence $(b_k)\in \Chat^{\N}$  such that  $b_k \in g_1^{-2n_k}(a_{n_k})$ and  $b_k \in V$. Hence, $\lim_k \Vert (g_1^{2n_k})'(b_k)\Vert  \le\lim_k c^{n_k}=0$.   Moreover, write  $g_{1}^{2}=f_\alpha$,  for some $m\in\N$ and $\alpha\in I^{m}$, and denote by $\alpha^n:=(\alpha \dots \alpha)\in I^{mn}$ the $n$-fold concatenation of $\alpha$. Let $(\beta_k)\in I^\N$ with $(\beta_k, a_{n_k})\in J(\tilde{f})$.  Then $\big(\alpha^{n_k}\beta_k, b_k) \in J(\tilde{f})$.  This contradicts that $G$ is expanding and finishes the proof.\end{proof}
\begin{lem}
\label{lem:moebiussemigroup-expanding-implies-J-not-twoelements}
Let \foreignlanguage{english}{\textup{$G=\left\langle f_{i}:i\in I\right\rangle $}}
denote an expanding rational semigroup. Suppose that $1\le \card\left(J\left(G\right)\right) \le 2$.
Then  we have $\card\left(J\left(G\right)\right)=1$. \end{lem}
\begin{proof}
Clearly, we have $G\subset\Aut(\Chat)$ and  each element of $G$ is loxodromic by Proposition  \ref{prop:expandingness-criteria} (\ref{enu:exp-implies-hyperbolicloxodromic}).  Now, suppose by way of contradiction that $J\left(G\right)=\left\{ a,b\right\} $ with $a\neq b$. Without loss of generality, we may assume that $a=0$ and $b=\infty$. Since $J(G)$ is $G$-backward invariant, we have $g\left(a\right)=a$ and $g\left(b\right)=b$ for each $g\in G$. Thus, there exists a sequence $\left(c_{i}\right)\in\C^{I}$ such that $f_{i}\left(z\right)=c_{i}z$, for each $z\in\Chat$ and $i\in I$. We may assume that there exists $i_0 \in I$ such that $\Vert f_{i_0}'(a) \Vert >1$. Since $G$ is expanding with respect to $\left\{ f_{i}:i\in I\right\} $, there exists a constant $c_0>1$ such that $\Vert f_{i}'\left(a\right)\Vert =\left|c_{i}\right|\ge c_0>1$,  for all $i\in I$. 
Hence, we have  $\Vert f_{i}'\left(b\right)\Vert  \le c_0^{-1}<1$, for all $i\in I$, which gives that $b\in F\left(G\right)$. This contradiction proves the lemma. 
\end{proof}

\section{\label{sec:Thermodynamic-formalism-for}Thermodynamic formalism for
expanding rational semigroups}

In this section we collect some of the main results from the thermodynamic
formalism in the context of expanding rational semigroups. It was
shown in \cite{MR2153926} that the skew product of a finitely generated
expanding rational semigroup is an open distance expanding map. We
refer to \cite{MR2656475} or the classical references \cite{bowenequilibriumMR0442989,MR511655,MR648108}
for general results on thermodynamic formalism for expanding maps.

\subsection{Conformal measure and equilibrium states}

We first give the definition of conformal measures which are useful
to investigate geometric properties of the Julia set of a rational
semigroup. A general notion of conformal measure was introduced in
\cite{MR1014246}. For results on conformal measures in the context
of expanding rational semigroups we refer to \cite{MR1625124,MR2153926}.
\begin{defn}
\label{conformal-measure}Let $f=\left(f_{i}\right)_{i\in I}\in\left(\Rat\right)^{I}$
and let $\tilde{\varphi}:J\left(\tilde{f}\right)\rightarrow\R$ be
Borel measurable. A Borel probability measure $\tilde{\nu}$ on $J\left(\tilde{f}\right)$
is called \emph{$\tilde{\varphi}$-conformal} (for $\tilde{f}$) if,
for each Borel set $A\subset J\left(\tilde{f}\right)$ such that $\tilde{f}\big|_{A}$
is injective, we have 
\[
\tilde{\nu}\bigl(\tilde{f}\left(A\right)\bigr)=\int_{A}\e^{-\tilde{\varphi}}d\tilde{\nu}.
\]

\end{defn}
Next we give the fundamental definitions of topological pressure and
equilibrium states.
\begin{defn}
Let $f=\left(f_{i}\right)_{i\in I}\in\left(\Rat\right)^{I}$ and let
$\tilde{\varphi}:J\left(\tilde{f}\right)\rightarrow\R$ be continuous.
The\emph{ topological pressure} $\mathcal{P}\left(\tilde{\varphi},\tilde{f}\right)$
of $\tilde{\varphi}$ with respect to $\tilde{f}: J(\tilde{f})\rightarrow J(\tilde{f})$ is given by 
\[
\mathcal{P}\left(\tilde{\varphi},\tilde{f}\right):=\sup\left\{ h\left(\tilde{m}\right)+\int\tilde{\varphi}d\tilde{m}:\tilde{m}\in\mathcal{M}_{e}\left(\tilde{f}\right)\right\} ,
\]
where $\mathcal{M}_{e}\left(\tilde{f}\right)$ denotes the set of
all $\tilde{f}$-invariant ergodic Borel probability measures on $J\left(\tilde{f}\right)$
and $h\left(\tilde{m}\right)$ refers to the measure-theoretic entropy
of $\left(\tilde{f},\tilde{m}\right)$. An ergodic $\tilde{f}$-invariant
Borel probability measure $\tilde{\mu}$ on $J\left(\tilde{f}\right)$
is called an \emph{equilibrium state} for $\tilde{\varphi}$ if 
\[
\mathcal{P}\left(\tilde{\varphi},\tilde{f}\right)=h(\tilde{\mu})+\int\tilde{\varphi}d\tilde{\mu}.
\]

\end{defn}
The following lemma guarantees existence and uniqueness of conformal
measures and equilibrium states for H\"older continuous potentials.
The lemma can be proved similarly as in \cite[Lemma 3.6 and  3.10]{MR2153926}.
For the uniqueness of the equilibrium state, see e.g. \cite{MR2656475}. 

For $\omega,\kappa\in I^{\N}$, we set $d_{0}\left(\omega,\tau\right):=2^{-\left|\omega\wedge\kappa\right|}$,
where $\omega\wedge\kappa$ denotes the longest common initial block
of $\omega$ and $\kappa$. The Julia set $J\left(\tilde{f}\right)$
is equipped with the metric $\tilde{d}$ which is for $\left(\omega,x\right),\left(\kappa,y\right)\in J\left(\tilde{f}\right)$
given by $\tilde{d}\left(\left(\omega,x\right),\left(\kappa,y\right)\right):=d_{0}\left(\omega,\kappa\right)+d\left(x,y\right)$,
where $d$ denotes the spherical distance on $\Chat$. We say that
$\tilde{\varphi}:J\left(\tilde{f}\right)\rightarrow\R$ is H\"older
continuous if there exists $\theta>0$ such that 
\[
\sup\left\{ \frac{d\left(\tilde{\varphi}\left(\omega,x\right),\tilde{\varphi}\left(\kappa,y\right)\right)}{\tilde{d}\left(\left(\omega,x\right),\left(\kappa,y\right)\right)^{\theta}}:\left(\omega,x\right),\left(\kappa,y\right)\in J\left(\tilde{f}\right),\left(\omega,x\right)\neq\left(\kappa,y\right)\right\} <\infty.
\]

\begin{lem}
\label{lem:existence-skew-product-conformalmeasure}Let $f=\left(f_{i}\right)_{i\in I}\in\left(\Rat\right)^{I}$
be expanding. Let $\tilde{\varphi}:J\left(\tilde{f}\right)\rightarrow\R$
be H\"older continuous with $\mathcal{P}\left(\tilde{\varphi},\tilde{f}\right)=0$.
Then we have the following.
\begin{enumerate}
\item There exists a unique $\tilde{\varphi}$-conformal measure $\tilde{\nu}$
on $J\left(\tilde{f}\right)$. 
\item There exists a unique continuous function $\tilde{h}:J\left(\tilde{f}\right)\rightarrow\R^{+}$
such that the probability measure $\tilde{\mu}:=\tilde{h}d\tilde{\nu}$
is $\tilde{f}$-invariant. Moreover, we have that $\tilde{\mu}$ is
exact (hence ergodic) and \foreignlanguage{english}{\textup{$\tilde{\mu}$}}
is the unique equilibrium state for $\tilde{\varphi}$. 
\end{enumerate}
\end{lem}
We also consider subconformal measures on $J\left(G\right)$. 
\begin{defn}
Let $f=\left(f_{i}\right)_{i\in I}\in\left(\Rat\right)^{I}$ and let
$G=\left\langle f_{i}:i\in I\right\rangle $. Let $\varphi=\left(\varphi_{i}:f_{i}^{-1}\left(J\left(G\right)\right)\rightarrow\R\right)_{i\in I}$
be a family of measurable functions. A Borel probability measure $\nu$
on $J\left(G\right)$ is called \emph{$\varphi$-subconformal} (for
$f$) if, for each $i\in I$ and for each Borel set $B\subset f_{i}^{-1}\left(J\left(G\right)\right)$,
\[
\nu\left(f_{i}\left(B\right)\right)\le\int_{B}\e^{-\varphi_{i}}d\nu.
\]
 
\end{defn}
Next lemma shows that the support of a subconformal measure is equal
to the Julia set.
\begin{lem}
\label{lem:support-subconformal-contains-juliaset}Let $f=\left(f_{i}\right)_{i\in I}\in\left(\Rat\right)^{I}$
be expanding and let $G=\left\langle f_{i}:i\in I\right\rangle $.
Let $\nu$ denote a $\varphi$-subconformal measure of a measurable
family $\varphi=\left(\varphi_{i}:f_{i}^{-1}\left(J\left(G\right)\right)\rightarrow\R\right)_{i\in I}$.
Then we have $\supp\left(\nu\right)=J\left(G\right)$.\end{lem}
\begin{proof}
We consider two cases. If $\card\left(J\left(G\right)\right)\ge3$
then $E\left(G\right)\subset F\left(G\right)$ by Lemma \ref{lem:moebiussemigroup-expanding-implies-goodexceptionalset}.
Then $\supp\left(\nu\right)=J\left(G\right)$ can be proved similarly
as in \cite[Proposition 4.3]{MR1625124}. Finally, if $1\le\card\left(J\left(G\right)\right)\le2$,
then we have $\card\left(J\left(G\right)\right)=1$ by Lemma \ref{lem:moebiussemigroup-expanding-implies-J-not-twoelements},
which immediately gives that $\supp\left(\nu\right)=J\left(G\right)$.
The proof is complete. 
\end{proof}
For a Borel measure $\tilde{m}$ on $J\left(\tilde{f}\right)$ we
denote by $\left(\pi_{\Chat}\right)_{*}\left(\tilde{m}\right)$ the
pushforward measure, which is for each Borel set $B\subset J\left(G\right)$
given by $\left(\pi_{\Chat}\right)_{*}\left(\tilde{m}\right)\left(B\right):=\tilde{m}\bigl(\pi_{\Chat}^{-1}\left(B\right)\bigr)$.
Next lemma is a straightforward generalisation of \cite[Lemma 3.11]{MR2153926}.
\begin{lem}
\label{lem:existence-sub-conformalmeasure-on-juliaset}Let $f=\left(f_{i}\right)_{i\in I}\in\left(\Rat\right)^{I}$
be expanding and let $G=\left\langle f_{i}:i\in I\right\rangle $.
Let $\varphi=\left(\varphi_{i}\right)_{i\in I}$ be a H\"older family.
Suppose that $\mathcal{P}\left(\tilde{\varphi},\tilde{f}\right)=0$
and let $\tilde{\nu}$ denote the unique $\tilde{\varphi}$-conformal
measure. Then the probability measure $\nu:=\left(\pi_{\Chat}\right)_{*}\left(\tilde{\nu}\right)$
is a $\varphi$-subconformal measure with $\supp\left(\nu\right)=J(G)$. 
\end{lem}

\subsection{The free energy function}

Let us now introduce the free energy function and an important family
of associated measures (see \cite{MR511655} and \cite{bowenequilibriumMR0442989,MR1435198,pesindimensiontheoryMR1489237}).
\begin{defn}
\label{def-betaconformal-eigenmeasures}Let $f=\left(f_{i}\right)_{i\in I}\in\left(\Rat\right)^{I}$
be expanding and let $G=\left\langle f_{i}:i\in I\right\rangle $.
Let $\psi=\left(\psi_{i}\right)_{i\in I}$ be a H\"older family associated
with $f$. The \emph{free energy function }\foreignlanguage{english}{\emph{for
 $\left(f,\psi\right)$} is the unique function} $t:\R\rightarrow\R$
such that $\mathcal{P}\bigl(\beta\tilde{\psi}+t\left(\beta\right)\tilde{\zeta},\tilde{f}\bigr)=0$,
for each $\beta\in\R$. For each $\beta\in\R$, we denote by $\tilde{\nu}_{\beta}$
the unique $\beta\tilde{\psi}+t\left(\beta\right)\tilde{\zeta}$-conformal
measure for $\tilde{f}$, and we denote by $\tilde{\mu}_{\beta}$
the unique equilibrium state for $\beta\tilde{\psi}+t\left(\beta\right)\tilde{\zeta}$.
Moreover, we denote by $\nu_{\beta}$ the pushforward measure $\left(\pi_{\Chat}\right)_{*}\left(\tilde{\nu}_{\beta}\right)$
supported on $J\left(G\right)$. We also set 
\[
\alpha_{0}\left(\psi\right):=\frac{\int\tilde{\psi}d\tilde{\mu}_{0}}{\int\tilde{\zeta}d\tilde{\mu}_{0}}.
\]
\end{defn}
\begin{rem}
Using that $f$ is expanding, one immediately verifies that, for each
$\beta\in\R$, there exists a unique $t\left(\beta\right)$ such that
$\mathcal{P}\left(\beta\tilde{\psi}+t\left(\beta\right)\tilde{\zeta},\tilde{f}\right)=0$.
In particular, we have that $t\left(0\right)$ is the unique real
number $\delta$ such that $\mathcal{P}\left(\delta\tilde{\zeta},\tilde{f}\right)=0$,
which is also called the critical exponent of $f$ (\cite{MR2153926}).
Since we always assume that $f$ is non-exceptional, we have that
$t\left(0\right)=\delta>0$. 
\end{rem}
The following two propositions go back to work of  Ruelle (\cite{MR511655})
for shift spaces. Since the skew product map $\tilde{f}$ is an open
distance expanding map (\cite{MR2153926}), we see that $\left(J\left(\tilde{f}\right),\tilde{f}\right)$
is semi-conjugate to a shift space by choosing a Markov partition.
Moreover, by \cite[Proposition 3.2 (f)]{MR1767945} and 
Lemma~\ref{lem:moebiussemigroup-expanding-implies-goodexceptionalset}, 
$\tilde{f}:J(\tilde{f})\rightarrow J(\tilde{f})$ is topologically exact. 
Then it is standard to derive the following two propositions. (See
also \cite{MR1435198,pesindimensiontheoryMR1489237}, where these
results of Ruelle are applied to a similar kind of multifractal analysis
for conformal repellers.)
\begin{prop}
\label{prop:free-energy-function}Let $f=\left(f_{i}\right)_{i\in I}\in\left(\Rat\right)^{I}$
be expanding. Let $\psi=\left(\psi_{i}\right)_{i\in I}$ be a H\"older
family associated with $f$. The free energy function $t:\R\rightarrow\R$
for $\left(f,\psi\right)$ is convex and real analytic and its first
derivative is given by 
\[
t'\left(\beta\right)=-\frac{\int\tilde{\psi}d\tilde{\mu}_{\beta}}{\int\tilde{\zeta}d\tilde{\mu}_{\beta}},\quad\mbox{for each }\beta\in\R.
\]

\end{prop}
The following proposition gives a criterion for strict convexity of
the free energy function. For the readers, we give a proof. 
\begin{prop}
\label{prop:characterisation-strict-convexity}Let $f=\left(f_{i}\right)_{i\in I}\in\left(\Rat\right)^{I}$
be expanding. Let $\psi=\left(\psi_{i}\right)_{i\in I}$ be a H\"older
family associated with $f$ and let $t:\R\rightarrow\R$ denote the
free energy function for $\left(f,\psi\right)$. Suppose that there
exists $\gamma\in\R$ such that $\mathcal{P}\left(\gamma\tilde{\psi},\tilde{f}\right)=0$. Then, the following (1)-(4) are equivalent:
\begin{enumerate}
\item There exists  $\beta _{0}\in \R $ such that $t''(\beta _{0})=0$.  
\item $t'$ is constant on $\R .$
In this case, we have $t(\beta )=\delta -\beta \delta /\gamma$. 
\label{enu:linear}
\item $\tilde{\mu }_{0}=\tilde{\mu }_{\gamma }$. 
\item There exists a continuous function $v:J(\tilde{f})\rightarrow \R $
such that $\delta \tilde{\zeta} =\gamma \tilde{\psi}+v-v\circ \tilde{f}.$
\end{enumerate}
In particular, $t''>0$ on $\R $ if and only if
$\tilde{\mu }_{0}\neq \tilde{\mu }_{\gamma }.$

Moreover, if there exists $\beta \in \R $ such that
$t''(\beta )\neq 0$, then $t''>0$ on $\R $ and $t$ is strictly convex
on $\R$.
\end{prop}
\begin{proof}
$(J(\tilde{f}), \tilde{f})$ is an open distance expanding map and it is topologically exact. 
Fix a Markov partition $\{ R_{1},\ldots ,R_{d}\} $ as in \cite[p118]{MR2656475}. 
Let $A$ be a $d\times d$ matrix with $a_{ij}=0$ or $1$ according to 
$\tilde{f}(\mbox{Int}(R_{i}))\cap \mbox{Int}(R_{j})$ is empty or not.  
By \cite[Theorem 4.5.7]{MR2656475}, there exists a surjective H\"older continuous map 
$\pi :\Sigma \rightarrow J(\tilde{f})$ (which is almost bijective), such that 
$\tilde{f}\circ \pi =\pi \circ \sigma $, where 
$\sigma :\Sigma \rightarrow \Sigma $ is the subshift of finite type constructed by $A.$ 
By \cite[Theorem 4.5.8]{MR2656475}, every H\"older continuous function 
$\tilde{\varphi }$ on $J(\tilde{f})$ defines a H\"older continuous function  
$\tilde{\varphi }\circ \pi $ on $\Sigma $, and we have $P(\tilde{\varphi }, \tilde{f})=
P(\tilde{\varphi }\circ \pi ,\sigma ).$ Hence, the free energy function $t:\R \rightarrow \R $ 
is given by $P(\beta \tilde{\psi }\circ \pi +t(\beta )\tilde{\zeta }\circ \pi , \sigma )=0.$ 

We first show that (1) implies (2) and (3). Suppose that there exists $\beta _{0}\in \R $ 
such that $t''(\beta _{0})=0.$ 
Then it is well-known  (see e.g. \cite[p129]{MR2003772} where the 
geometric potential is given by $-\zeta $ in our notation) that 
there exists a  H\"older continuous function $u:\Sigma \rightarrow \R $ 
such that $t'(\beta _{0})\tilde{\zeta }\circ \pi +\tilde{\psi }\circ \pi =u-u\circ \sigma .$ 
It follows that $P((-\beta t'(\beta _{0})+t(\beta ))\tilde{\zeta }\circ \pi , \sigma )=0$. 
Hence, $-\beta t'(\beta _{0})+t(\beta )=t(0)=\delta .$ Thus, 
$t(\beta )=\delta +\beta t'(\beta _{0}).$ Therefore $t'(\beta )=t'(\beta _{0})$ for all $\beta \in \R .$ 
To determine $t'(\beta _{0})$, note that by the assumption there exists a unique 
$\gamma \in \R $ such that $P(\gamma \tilde{\psi }\circ \pi ,\sigma )=0.$ Hence, 
$t(\gamma )=0.$ It follows that $t'(\beta _{0})=-\delta /\gamma .$ 
We have thus shown that $t(\beta )=\delta -\beta \delta /\gamma $ and that 
$\delta \tilde{\zeta }\circ \pi $ is cohomologous to $\gamma \tilde{\psi }\circ \pi .$ 
Hence, $\rho _{\delta \tilde{\zeta }\circ \pi }=\rho _{\gamma \tilde{\psi }\circ \pi }$, 
where $\rho _{\delta \tilde{\zeta }\circ \pi }$ and $\rho _{\gamma \tilde{\psi }\circ \pi }$ 
are the unique equilibrium states of $\delta \tilde{\zeta }\circ \pi $ and 
$\tilde{\psi }\circ \pi $, respectively. 
Since $\pi $ defines an isomorphism of the probability spaces 
$(\Sigma, \rho _{\delta \tilde{\zeta }\circ \pi })$ and 
$(J(\tilde{f}), \rho _{\delta \tilde{\zeta }\circ \pi }\circ \pi ^{-1})$ by 
\cite[Theorem 4.5.9]{MR2656475}, 
we have that $\rho _{\delta \tilde{\zeta }\circ \pi }\circ \pi ^{-1}$ is the equilibrium 
state for $\delta \tilde{\zeta }.$ Similarly, 
$\rho _{\gamma \tilde{\psi }\circ \pi }\circ \pi ^{-1}$ is the equilibrium state 
for  $\gamma \tilde{\psi }.$ 
Also, 
$\rho _{\delta \tilde{\zeta }\circ \pi }\circ \pi ^{-1}=
\rho _{\gamma \tilde{\psi }\circ \pi }\circ \pi ^{-1}. $  
 Since the equilibrium states of   H\"older  continuous potentials on $J(\tilde{f})$ are 
 unique (\cite[Theorem 5.6.2]{MR2656475}) 
 and the probability measures $\tilde{\mu }_{0}$ and $\tilde{\mu }_{\gamma }$ are 
 equilibrium states for $\delta \tilde{\zeta }$ and $\gamma \tilde{\zeta }$ 
 respectively, we conclude that $\rho _{\delta \tilde{\zeta }\circ \pi }\circ \pi ^{-1}
 =\tilde{\mu }_{0}$ and $\rho _{\gamma \tilde{\psi }\circ \pi }\circ \pi ^{-1}=\tilde{\mu}_{\gamma }.$ 
Thus it follows that $\tilde{\mu }_{0}=\tilde{\mu }_{\gamma }.$  Hence, we have shown that (1) implies (2) and (3). 
 
We now suppose (3).  To prove (4) we
proceed similarly as in the proof of \cite[Theorem 3.1]{su09}. We
consider the Perron-Frobenius operators 
\[
\mathcal{L}_{0}:C\left(J\left(\tilde{f}\right)\right)\rightarrow C\left(J\left(\tilde{f}\right)\right)\quad\mbox{and}\quad\mathcal{L}_{\gamma}:C\left(J\left(\tilde{f}\right)\right)\rightarrow C\left(J\left(\tilde{f}\right)\right),
\]
which are for $\tilde{h}\in C\left(J\left(\tilde{f}\right)\right)$
and $\left(\omega,z\right)\in J\left(\tilde{f}\right)$ given by 
\[
\mathcal{L}_{0}\left(\tilde{h}\right)\left(\omega,z\right):=\sum_{\left(i\omega,y\right)\in\tilde{f}^{-1}\left(\omega,z\right)}\Vert\tilde{f}'_{i}\left(y\right)\Vert^{-\delta}\tilde{h}\left(i\omega,y\right)\quad\mbox{and}\quad\mathcal{L}_{\gamma}\left(\tilde{h}\right)\left(\omega,z\right):=\sum_{\left(i\omega,y\right)\in\tilde{f}^{-1}\left(\omega,z\right)}\e^{\gamma\psi_{i}\left(y\right)}\tilde{h}\left(i\omega,y\right).
\]
We denote by $\mathcal{L}_{0}^{*}$ and $\mathcal{L}_{\gamma}^{*}$
the dual operators acting on the space $\left(C\left(J\left(\tilde{f}\right)\right)\right)^{*}$
of bounded linear functionals on $C\left(J\left(\tilde{f}\right)\right)$.
Then it follows from \cite{MR2153926} that there exist unique continuous 
positive 
functions $\tilde{h}_{0},\tilde{h}_{\gamma}:J\left(\tilde{f}\right)\rightarrow\R^{+}$
such that $\mathcal{L}_{0}\left(\tilde{h}_{0}\right)=\tilde{h}_{0}$
and $\mathcal{L}_{\gamma}\left(\tilde{h}_{\gamma}\right)=\tilde{h}_{\gamma}$,
and unique Borel probability measures $\tilde{\nu}_{0},\tilde{\nu}_{\gamma}$
on $J\left(\tilde{f}\right)$ such that $\mathcal{L}_{0}^{*}\left(\tilde{\nu}_{0}\right)=\tilde{\nu}_{0}$
and $\mathcal{L}_{\gamma}^{*}\left(\tilde{\nu}_{\gamma}\right)=\tilde{\nu}_{\gamma}$.
Then $\tilde{\nu}_{0}$ is $\delta\tilde{\zeta}$-conformal and $\tilde{\nu}_{\gamma}$
is $\gamma\tilde{\psi}$-conformal (see Definition \ref{conformal-measure}
and \cite{MR1014246}). Moreover, we have that the unique equilibrium
states $\tilde{\mu}_{0},\tilde{\mu}_{\gamma}$ are given by $\tilde{\mu}_{0}=\tilde{h}_{0}d\tilde{\nu}_{0}$
and $\tilde{\mu}_{\gamma}=\tilde{h}_{\gamma}d\tilde{\nu}_{\gamma}$
(see also Lemma \ref{lem:existence-skew-product-conformalmeasure}).
Since $\tilde{\mu}_{0}=\tilde{\mu}_{\gamma}$, we have that $\tilde{h}_{0}d\tilde{\nu}_{0}=\tilde{h}_{\gamma}d\tilde{\nu}_{\gamma}$.
Using this equality and conformality of $\tilde{\nu}_{0}$ and $\tilde{\nu}_{\gamma}$,
we obtain that, for each Borel set $A\subset J\left(\tilde{f}\right)$
such that $\tilde{f}_{|A}$ is injective, 
\[
\tilde{\mu}_{0}\left(\tilde{f}\left(A\right)\right)=\int_{\tilde{f}\left(A\right)}\tilde{h}_{\gamma}d\tilde{\nu}_{\gamma}=\int_{A}\e^{-\gamma\tilde{\psi}}\left(\tilde{h}_{\gamma}\circ\tilde{f}\right)d\tilde{\nu}_{\gamma}=\int_{A}\e^{-\gamma\tilde{\psi}}\left(\tilde{h}_{\gamma}\circ\tilde{f}\right)\frac{\tilde{h}_{0}}{\tilde{h}_{\gamma}}d\tilde{\nu}_{0}
\]
and  
\[
\tilde{\mu}_{0}\left(\tilde{f}\left(A\right)\right)=\left(\tilde{h}_{0}d\tilde{\nu}_{0}\right)\left(\tilde{f}\left(A\right)\right)=\int_{A}\Vert\tilde{f}'\Vert^{\delta}\left(\tilde{h}_{0}\circ\tilde{f}\right)d\tilde{\nu}_{0}.
\]
We deduce that 
\[
\left(\tilde{h}_{0}\circ\tilde{f}\right)\cdot\Vert\tilde{f}'\Vert^{\delta}=\e^{-\gamma\tilde{\psi}}\frac{\tilde{h}_{\gamma}\circ\tilde{f}}{\tilde{h}_{\gamma}}\tilde{h}_{0}.
\]
By taking logarithm, we have thus shown that 
\begin{equation}
-\delta\tilde{\zeta}=-\gamma\tilde{\psi}+\log\tilde{h}_{\gamma}\circ\tilde{f}-\log\tilde{h}_{\gamma}+\log\tilde{h}_{0}-\log\tilde{h}_{0}\circ\tilde{f},\label{eq:cohomologous-potentials}
\end{equation}
 which proves (4). 

We now suppose (4).  Then  $P((\beta +t(\beta )\gamma /\delta )\tilde{\psi }, \tilde{f})=0.$ 
 It follows that $t(\beta )=\delta -\beta \delta /\gamma .$ Thus,  
 $t'$ is constant on $\R$, which proves (1) and thus finishes the proof of the proposition.  
\end{proof}

Recall that, for a convex function $a:\R\rightarrow\R$, its convex
conjugate $a^{*}:\R\rightarrow\R\cup\left\{ \infty\right\} $ is given
by \foreignlanguage{english}{$a^{*}\left(c\right):=\sup_{\beta\in\R}\left\{ \beta c-a\left(\beta\right)\right\} $
for each $c\in\R$.} We  make use of the following facts about the
convex conjugate. We refer to \cite[Theorem 23.5, 26.5]{rockafellar-convexanalysisMR0274683}
for the proofs and further details. 
\begin{lem}
\label{lem:conjugate}Let $a:\R\rightarrow\R$ be convex and differentiable
and let $a^{*}:\R\rightarrow\R\cup\left\{ \infty\right\} $ be the
convex conjugate of $a$.  
\begin{enumerate}
\item \label{enu:conjugate-facts-1}For each $\beta\in\R$ we have that
$a^{*}\left(a'\left(\beta\right)\right)=a'\left(\beta\right)\beta-a\left(\beta\right)$. 
\item \label{enu:conjugate-facts-2}If $\alpha\notin\overline{a'\left(\R\right)}$
then $a^{*}\left(\alpha\right)=\infty$. 
\item \label{enu:conjugate-facts-3}Suppose that $a$ is twice differentiable and 
$a''>0$ on $\Bbb{R}.$ 
Then $a'\left(\R\right)$ is an open subset of $\R$, $a':\R\rightarrow a'\left(\R\right)$
is invertible, $a^{*}$ is twice differentiable on $a'\left(\R\right)$, 
$(a^{\ast })''>0$ on $a'\left(\R\right)$, 
and we have $\left(a^{*}\right)'\left(a'\left(\beta\right)\right)=\beta$
for each $\beta\in\R$. If moreover $a$ is real analytic, then $a^{*}$
is real analytic on $a'\left(\R\right)$. 
\end{enumerate}
\end{lem}
The proofs of the following two lemmata are standard (see e.g. \cite{MR1435198,pesindimensiontheoryMR1489237}).
To make this article more self-contained, we include the proofs. In
the next lemma we deduce analytic properties of the convex conjugate
of the free energy function. 
\begin{lem}
\label{lem:legendre-transform-convex-conjugate}Let $f=\left(f_{i}\right)_{i\in I}\in\left(\Rat\right)^{I}$
be expanding and let $\delta>0$ denote the critical exponent of \foreignlanguage{english}{\textup{$f$.}}
Let $\psi=\left(\psi_{i}\right)_{i\in I}$ be a H\"older family associated
with $f$ and let $t:\R\rightarrow\R$ denote the free energy function
for $\left(f,\psi\right)$. Suppose that there exists $\gamma\in\R$
such that $\mathcal{P}\left(\gamma\tilde{\psi},\tilde{f}\right)=0$.
Let $\beta\in\R$ and $\alpha=-t'\left(\beta\right)$. Then we have
\[
-t^{*}\left(-\alpha\right)=\beta\alpha+t\left(\beta\right)=-\frac{h\left(\tilde{\mu}_{\beta}\right)}{\int\tilde{\zeta}d\tilde{\mu}_{\beta}}>0.
\]
If $\tilde{\mu}_{0}=\tilde{\mu}_{\gamma}$ then we have $-t^{*}\left(-\alpha_{0}(\psi )\right)=\delta$
and $-t^{*}\left(-\alpha\right)=-\infty$ for each $\alpha\neq\alpha_{0}(\psi )$.
If $\tilde{\mu}_{0}\neq\tilde{\mu}_{\gamma}$ then $s(\alpha ):=-t^{*}\left(-\alpha \right)$
is a strictly concave real analytic function on $-t'\left(\R\right)$
with maximum value $\delta=-t^{*}\left(-\alpha_{0}(\psi )\right)>0$, and 
$s''<0$ on $-t'\left(\R\right)$.
\end{lem}
\begin{proof}
That $-t^{*}\left(-\alpha\right)=\beta\alpha+t\left(\beta\right)$
follows from Lemma \ref{lem:conjugate} (\ref{enu:conjugate-facts-1}).
Since $\tilde{\mu}_{\beta}$ is the equilibrium state for $\beta\tilde{\psi}+t\left(\beta\right)\tilde{\zeta}$
and $\mathcal{P}\bigl(\beta\tilde{\psi}+t\left(\beta\right)\tilde{\zeta},\tilde{f}\bigr)=0$,
we have $-h\left(\tilde{\mu}_{\beta}\right)=\int\beta\tilde{\psi}+t\left(\beta\right)\tilde{\zeta}d\tilde{\mu}_{\beta}$.
Combining with Proposition \ref{prop:free-energy-function}, we obtain
$-t^{*}\left(-\alpha\right)=\beta\alpha+t\left(\beta\right)=-h\left(\tilde{\mu}_{\beta}\right)\big/\int\tilde{\zeta}d\tilde{\mu}_{\beta}$. 

To prove that $-h\left(\tilde{\mu}_{\beta}\right)\big/\int\tilde{\zeta}d\tilde{\mu}_{\beta}>0$
first observe that $-h\left(\tilde{\mu}_{\beta}\right)\big/\int\tilde{\zeta}d\tilde{\mu}_{\beta}\ge0$.
Now suppose for a contradiction that there exists $\beta_{0}\in\R$
such that $-h\left(\tilde{\mu}_{\beta_{0}}\right)\big/\int\tilde{\zeta}d\tilde{\mu}_{\beta_{0}}=-t^{*}\left(t'\left(\beta_{0}\right)\right)=0$.
We distinguish two cases according to Proposition \ref{prop:characterisation-strict-convexity}.
First suppose that $t'$ is constant on $\R$. Then we have $0=-t^{*}\left(t'\left(\beta_{0}\right)\right)=-t^{*}\left(t'\left(0\right)\right)=t\left(0\right)$
by Lemma \ref{lem:conjugate} (\ref{enu:conjugate-facts-1}). This
gives the desired contradiction, because $f$ is non-exceptional giving
that $t\left(0\right)=\delta>0$. For the remaining case, we may assume
that $t''>0$ on $\R .$ Since $-t^{*}\left(c\right)\ge0$ for
all $c$ in the open neighbourhood $t'\left(\R\right)$ of $t'\left(\beta_{0}\right)$
and $-t^{*}\left(t'\left(\beta_{0}\right)\right)=0$, we conclude
that the derivative of $t^{*}$ vanishes in $t'\left(\beta_{0}\right)$.
 By Lemma \ref{lem:conjugate} (\ref{enu:conjugate-facts-3}), it
follows that zero is a local maximum of $-t^{*}$ in a neighbourhood
of $t'\left(\beta_{0}\right)$, which implies that $-t^{*}$ is constant
in a neighbourhood of $t'\left(\beta_{0}\right)$. However, by Lemma
\ref{lem:conjugate} (\ref{enu:conjugate-facts-3}), we have that
$(t^{*})''>0$ on $t'\left(\R\right)$ which is a contradiction.
We have thus shown that $-h\left(\tilde{\mu}_{\beta}\right)\big/\int\tilde{\zeta}d\tilde{\mu}_{\beta}>0$
for all $\beta\in\R$. 

To verify the remaining assertions, first suppose that $\tilde{\mu}_{0}=\tilde{\mu}_{\gamma}$.
By Proposition \ref{prop:free-energy-function} and \ref{prop:characterisation-strict-convexity}
we then have that $t'\left(\beta\right)=t'\left(0\right)=-\alpha_{0}(\psi )$
for all $\beta\in\R$. By Lemma \ref{lem:conjugate} (\ref{enu:conjugate-facts-1})
and (\ref{enu:conjugate-facts-2}) we conclude that $-t^{*}\left(-\alpha_{0}(\psi )\right)=-t^{*}\left(t'\left(0\right)\right)=t\left(0\right)=\delta$
and $-t^{*}\left(-\alpha\right)=-\infty$ 
if $\alpha\neq\alpha_{0}(\psi )$.
Now suppose that $\tilde{\mu}_{0}\neq\tilde{\mu}_{\gamma}$. By Proposition
\ref{prop:free-energy-function} and \ref{prop:characterisation-strict-convexity},
we have that $t$ is strictly convex and real analytic. By Lemma \ref{lem:conjugate}
(\ref{enu:conjugate-facts-3}) we have that $-t^{*}$ is a strictly
concave and real analytic function on $-t'\left(\R\right)$. Also, 
$s''<0$ on $t'(\Bbb{R}).$ 
Moreover,
Lemma \ref{lem:conjugate} (\ref{enu:conjugate-facts-3}) implies
that the derivative of $t^{*}$ vanishes in 
$t'\left(0\right)=-\alpha_{0}(\psi )$,
which shows that $-t^{*}$ attains a maximum in $-\alpha_{0}(\psi )$ with
$-t^{*}\left(-\alpha_{0}(\psi )\right)=\delta$. 
\end{proof}
For the support of the measures $\tilde{\nu}_{\beta}$ and $\nu_{\beta}$
we prove the following lemma. Recall that $\mathcal{F}\left(\alpha,\psi\right)$
is the continuous image of the Borel set \foreignlanguage{english}{$\mathcal{\tilde{F}}\left(\alpha,\psi\right)$.
In particular, $\mathcal{F}\left(\alpha,\psi\right)$ is a Suslin
set and thus} $\nu_{\beta}$-measurable. We refer to \cite[p65-70]{MR0257325}
for details on Suslin sets. 
\begin{lem}
\label{lem:support-of-nu-beta}Let $f=\left(f_{i}\right)_{i\in I}\in\left(\Rat\right)^{I}$
be expanding. Let $\psi=\left(\psi_{i}\right)_{i\in I}$ be a H\"older
family associated with $f$ and let $t:\R\rightarrow\R$ denote the
free energy function for $\left(f,\psi\right)$. For each $\beta\in\R$,
we have that 
\[
\tilde{\nu}_{\beta}\left(\mathcal{\tilde{F}}\left(-t'\left(\beta\right),\psi\right)\right)=\nu_{\beta}\left(\mathcal{F}\left(-t'\left(\beta\right),\psi\right)\right)=1.
\]
In particular, for each $\alpha\in-t'\left(\R\right)$, we have that
$\mathcal{F}\left(\alpha,\psi\right)$ is non-empty.\end{lem}
\begin{proof}
Let $\beta\in\R$. Since $\tilde{\mu}_{\beta}$ is ergodic by Lemma
\ref{lem:existence-skew-product-conformalmeasure}, we have by Birkhoff's
ergodic theorem that for $\tilde{\mu}_{\beta}$-almost every $x\in J\left(\tilde{f}\right)$,
\[
\lim_{n\rightarrow\infty}\frac{S_{n}\tilde{\psi}\left(x\right)}{S_{n}\tilde{\zeta}\left(x\right)}=\frac{\int\tilde{\psi}d\tilde{\mu}_{\beta}}{\int\tilde{\zeta}d\tilde{\mu}_{\beta}}.
\]
Since $\int\tilde{\psi}d\tilde{\mu}_{\beta}\big/\int\tilde{\zeta}d\tilde{\mu}_{\beta}=-t'\left(\beta\right)$
by Proposition \ref{prop:free-energy-function} and $\tilde{\mu}_{\beta}$
and $\tilde{\nu}_{\beta}$ are equivalent by Lemma \ref{lem:existence-skew-product-conformalmeasure},
we have that $\tilde{\nu}_{\beta}\left(\mathcal{\tilde{F}}\left(-t'\left(\beta\right),\psi\right)\right)=1$.
Consequently, we have that $\nu_{\beta}\left(\mathcal{F}\left(-t'\left(\beta\right),\psi\right)\right)=1$. \end{proof}
\begin{rem}
\label{borelsest-of-fullmeasure}Under the assumptions of Lemma \ref{lem:support-of-nu-beta}
there exists a Borel measurable subset $A\subset\mathcal{F}\left(-t'\left(\beta\right),\psi\right)$
such that $\nu_{\beta}\left(A\right)=1$. \end{rem}
\begin{proof}
Since the Borel measure $\tilde{\nu}_{\beta}$ is regular, there exists
a family of compact subsets $\left(K_{n}\right)_{n\in\N}$, $K_{n}\subset\mathcal{\tilde{F}}\left(-t'\left(\beta\right),\psi\right)$,
such that $\tilde{\nu}_{\beta}\left(\mathcal{\tilde{F}}\left(-t'\left(\beta\right),\psi\right)\setminus\bigcup_{n\in\N}K_{n}\right)=0$.
Hence, for the Borel set $\bigcup_{n\in\N}\pi_{\Chat}\left(K_{n}\right)\subset\mathcal{F}\left(-t'\left(\beta\right),\psi\right)$,
we have that $\nu_{\beta}\left(\bigcup_{n\in\N}\pi_{\Chat}\left(K_{n}\right)\right)=1$. 
\end{proof}

\section{\label{sec:Multifractal-Formalism}Multifractal Formalism}

To establish the multifractal formalism for expanding multi-maps,
we investigate the local dimension of the measures $\nu_{\beta}$
for $\beta\in\R$ (see Definition \ref{def-betaconformal-eigenmeasures}).
To state the next lemma, we have to make a further definition. Recall
that $\alpha_{0}\left(\psi\right):=\int\tilde{\psi}d\tilde{\mu}_{0}\big/\int\tilde{\zeta}d\tilde{\mu}_{0}$
for a H\"older family $\psi$.
\begin{defn}
Let $f=\left(f_{i}\right)_{i\in I}\in\left(\Rat\right)^{I}$ be expanding.
Let $\psi=\left(\psi_{i}\right)_{i\in I}$ be a H\"older family associated
with $f$. For $\alpha\in\R$ we define
\[
\tilde{\mathcal{F}}^{\sharp}\left(\alpha,\psi\right):=\begin{cases}
\left\{ x\in J\left(\tilde{f}\right):\limsup_{n\rightarrow\infty}S_{n}\tilde{\psi}\left(x\right)\big/S_{n}\tilde{\zeta}\left(x\right)\ge\alpha\right\} , & \mbox{for }\alpha\ge\alpha_{0}\left(\psi\right),\\
\left\{ x\in J\left(\tilde{f}\right):\liminf_{n\rightarrow\infty}S_{n}\tilde{\psi}\left(x\right)\big/S_{n}\tilde{\zeta}\left(x\right)\le\alpha\right\} , & \mbox{for }\alpha<\alpha_{0}\left(\psi\right).
\end{cases}
\]
Moreover, we set $\mathcal{F}^{\sharp}\left(\alpha,\psi\right):=\pi_{\Chat}\left(\tilde{\mathcal{F}}^{\sharp}\left(\alpha,\psi\right)\right)$. 
\end{defn}
The proof of the following lemma mimics the proof of \cite[Theorem 3.4]{MR1625124},
where $\nu_{0}$ is considered.  To state the lemma, let $B\left(z,r\right)$
denote the spherical ball of radius $r$ centred at $z\in\Chat$.
Recall from Lemma \ref{lem:support-subconformal-contains-juliaset}
that $\nu_{\beta}\left(B\left(z,r\right)\right)>0$ for each $z\in J\left(G\right)$
and $r>0$. Thus, we have that $\log\nu_{\beta}\left(B\left(z,r\right)\right)$
is a well-defined real number.
\begin{lem}
\label{lem:upper-bound}Let $f=\left(f_{i}\right)_{i\in I}\in\left(\Rat\right)^{I}$
be expanding and let \foreignlanguage{english}{\textup{$G=\left\langle f_{i}:i\in I\right\rangle $.}}
Let $\psi=\left(\psi_{i}\right)_{i\in I}$ be a H\"older family associated
with $f$ and let $t:\R\rightarrow\R$ denote the free energy function
for $\left(f,\psi\right)$. Let $\alpha,\beta\in\R$ such that, either
(1) $\alpha\ge\alpha_{0}\left(\psi\right)$ and $\beta\le0$, or (2)
$\alpha\le\alpha_{0}\left(\psi\right)$ and $\beta\ge0$. For each
$z\in\mathcal{F}^{\sharp}\left(\alpha,\psi\right)$ we then have that
\[
0\le\liminf_{r\rightarrow0}\frac{\log\nu_{\beta}\left(B\left(z,r\right)\right)}{\log r}\le t\left(\beta\right)+\beta\alpha.
\]
\end{lem}
\begin{proof}
We  only consider the case that $\alpha\ge\alpha_{0}\left(\psi\right)$
and $\beta\le0$. The remaining case can be proved in a similar fashion.
Let $z\in\mathcal{F}^{\sharp}\left(\alpha,\psi\right)$. There exists
$\omega\in I^{\N}$ such that $\left(\omega,z\right)\in\tilde{\mathcal{F}}^{\sharp}\left(\alpha,\psi\right)$.
Since $\alpha\ge\alpha_{0}\left(\psi\right)$ we have 
\begin{equation}
\limsup_{n\rightarrow\infty}\frac{S_{n}\tilde{\psi}\left(\left(\omega,z\right)\right)}{S_{n}\tilde{\zeta}\left(\left(\omega,z\right)\right)}\ge\alpha.\label{eq:proof-upperbound-1}
\end{equation}
Since $\left(\omega,z\right)\in J\left(\tilde{f}\right)$ we have
$z\in J_{\omega}$ by Lemma \ref{lem:exp-implies-compactfibres}.
 We set $z_{n}:=f_{\omega_{|n}}\left(z\right)$ for each $n\in\N$.
By Proposition \ref{prop:basic-facts} we have $z_{n}\in J\left(G\right)$.
Since $f$ is expanding, we have that $G$ is hyperbolic and there
exists a $G$-forward invariant non-empty compact subset of $F\left(G\right)$
by Proposition \ref{prop:expandingness-criteria} (\ref{enu:exp-implies-hyperbolicloxodromic}).
Hence, there exists $R>0$ such that, for each $n\in\N$, there exists
a holomorphic branch $\phi_{n}:B\left(z_{n},R\right)\rightarrow\Chat$
of $f_{\omega_{|n}}^{-1}$ such that $f_{\omega_{|n}}\circ\phi_{n}=\id\big|_{B\left(z_{n},R\right)}$
and $\phi\bigl(f_{\omega_{|n}}\left(z\right)\bigr)=z$. By Koebe's
distortion theorem, there exist constants $c_{1}>0$ and $c_{2}>1$
such that for each $n\in\N$, 
\begin{equation}
\phi_{n}\left(B\left(z_{n},c_{2}^{-1}R\right)\right)\subset B\left(z,c_{1}^{-1}R\left\Vert \phi_{n}'\left(z_{n}\right)\right\Vert \right).\label{eq:proof-upperbound-2}
\end{equation}
Using that $\nu_{\beta}$ is $\beta\psi+t\left(\beta\right)\zeta$-subconformal
by Lemma \ref{lem:existence-sub-conformalmeasure-on-juliaset} and
the set inclusion in (\ref{eq:proof-upperbound-2}), we obtain that
for each $n\in\N$, 
\begin{eqnarray*}
\nu_{\beta}\left(f_{\omega_{|n}}\left(\phi_{n}\left(B\left(z_{n},c_{2}^{-1}R\right)\right)\right)\right) & \le & \int_{\phi_{n}\left(B\left(z_{n},c_{2}^{-1}R\right)\right)}\e^{-S_{n}\left(\beta\tilde{\psi}+t\left(\beta\right)\tilde{\zeta}\right)\left(\omega,\cdot\right)}d\nu_{\beta}\\
 & \le & \nu_{\beta}\left(B\left(z,c_{1}^{-1}R\left\Vert \phi_{n}'\left(z_{n}\right)\right\Vert \right)\right)\max_{y\in\phi_{n}\left(B\left(z_{n},c_{2}^{-1}R\right)\right)}\e^{-S_{n}\left(\beta\tilde{\psi}+t\left(\beta\right)\tilde{\zeta}\right)\left(\omega,y\right)}.
\end{eqnarray*}
Since $\supp\left(\nu_{\beta}\right)=J\left(G\right)$ by Lemma \ref{lem:support-subconformal-contains-juliaset}
and by the compactness of $J\left(G\right)$, there exists a constant
$M>0$ such that \foreignlanguage{english}{$\nu_{\beta}\bigl(f_{\omega_{|n}}\left(\phi_{n}\left(B\left(z_{n},c_{2}^{-1}R\right)\right)\right)\bigr)>M$
for all $n\in\N$.} Using that $f$ is expanding and that $\beta\tilde{\psi}+t\left(\beta\right)\tilde{\zeta}$
is H\"older continuous, one verifies that there exists a constant
$C>1$ such that, for all $n\in\N$ and for all $x\in\phi_{n}\left(B\left(z_{n},c_{2}^{-1}R\right)\right)$,
\[
\left|S_{n}\left(\beta\tilde{\psi}+t\left(\beta\right)\tilde{\zeta}\right)\left(\omega,x\right)-S_{n}\left(\beta\tilde{\psi}+t\left(\beta\right)\tilde{\zeta}\right)\left(\omega,z\right)\right|\le\log C.
\]
From $\beta\le0$ and (\ref{eq:proof-upperbound-1}) it follows that
there exists a sequence $\left(n_{j}\right)\in\N^{\N}$ tending to
infinity, such that, for each $\epsilon>0$, we have for all $j$
sufficiently large, 
\[
\max_{x\in\phi_{n_{j}}\bigl(B\bigl(z_{n_{j}},c_{2}^{-1}R\bigr)\bigr)}\e^{-S_{n_{j}}\left(\beta\tilde{\psi}+t\left(\beta\right)\tilde{\zeta}\right)\left(\omega,x\right)}\le C\e^{-S_{n_{j}}\tilde{\zeta}\left(\omega,z\right)\left(\beta\alpha+t\left(\beta\right)-\beta\epsilon\right)}=C\left\Vert \phi_{n_{j}}'\left(z_{n_{j}}\right)\right\Vert ^{-\left(\beta\alpha+t\left(\beta\right)-\beta\epsilon\right)}.
\]
We have thus shown that $0<M\le C\nu_{\beta}\left(B\left(z,c_{1}^{-1}R\left\Vert \phi_{n_{j}}'\left(z_{n_{j}}\right)\right\Vert \right)\right)\left\Vert \phi_{n_{j}}'\left(z_{n_{j}}\right)\right\Vert ^{-\left(\beta\alpha+t\left(\beta\right)-\beta\epsilon\right)}$
for all $j$ sufficiently large. Set $r_{j}:=c_{1}^{-1}R\left\Vert \phi_{n_{j}}'\left(z_{n_{j}}\right)\right\Vert $.
Clearly, we have that $\lim_{j}r_{j}=0$. Hence, we have 
\[
0\le\liminf_{r\rightarrow0}\frac{\log\nu_{\beta}\left(B\left(z,r\right)\right)}{\log r}\le\beta\alpha+t\left(\beta\right)-\beta\epsilon.
\]
Since $\epsilon$ was arbitrary, the proof is complete. \end{proof}
\begin{lem}
\label{lem:t-star-properties}Let $f=\left(f_{i}\right)_{i\in I}\in\left(\Rat\right)^{I}$
be expanding. Let $\psi=\left(\psi_{i}\right)_{i\in I}$ be a H\"older
family associated with $f$ and let $t:\R\rightarrow\R$ denote the
free energy function for $\left(f,\psi\right)$. Then we have the
following.
\begin{enumerate}
\item \label{enu:t-star-formula}For each $\alpha\in\R$ we have 
\[
-t^{*}\left(-\alpha\right)=\begin{cases}
\inf_{\beta\le0}\left\{ t\left(\beta\right)+\beta\alpha\right\} , & \mbox{for }\alpha\ge\alpha_{0}\left(\psi\right),\\
\inf_{\beta\ge0}\left\{ t\left(\beta\right)+\beta\alpha\right\} , & \mbox{for }\alpha\le\alpha_{0}\left(\psi\right).
\end{cases}
\]

\item \label{enu:ourside-range-empty}Let $\alpha\in\R$. If $-t^{*}\left(-\alpha\right)<0$
then $\mathcal{F}\left(\alpha,\psi\right)=\mathcal{F}^{\sharp}\left(\alpha,\psi\right)=\emptyset$.
In particular, we have that $\mathcal{F}\left(\alpha,\psi\right)=\mathcal{F}^{\sharp}\left(\alpha,\psi\right)=\emptyset$,
if $\alpha\notin-\overline{t'\left(\R\right)}$. 
\end{enumerate}
\end{lem}
\begin{proof}
To prove (\ref{enu:t-star-formula}), first recall that $\alpha_{0}\left(\psi\right)=-t'\left(0\right)$
by Proposition \ref{prop:free-energy-function}. Hence, we have that
$-t^{*}\left(-\alpha_{0}\left(\psi\right)\right)=-t^{*}\left(t'\left(0\right)\right)=t\left(0\right)$
by Lemma \ref{lem:conjugate} (\ref{enu:conjugate-facts-1}). Now
we only consider the case that $\alpha\ge\alpha_{0}\left(\psi\right)$.
The remaining case can be proved similarly. If $\beta>0$ then we
have that $t\left(\beta\right)+\beta\alpha\ge t\left(\beta\right)+\beta\alpha_{0}\left(\psi\right)\ge-t^{*}\left(-\alpha_{0}\left(\psi\right)\right)=t\left(0\right)$.
Hence, we have $-t^{*}\left(-\alpha\right)=\inf_{\beta\le0}\left\{ t\left(\beta\right)+\beta\alpha\right\} $. 

For the proof of (\ref{enu:ourside-range-empty}), suppose for a contradiction
that there exists $\alpha\in\R$ and $z\in\mathcal{F}^{\sharp}\left(\alpha,\psi\right)$
such that $-t^{*}\left(-\alpha\right)<0$. Again, we only consider
the case that $\alpha\ge\alpha_{0}\left(\psi\right)$. By (1) there
exists $\beta\le0$ such that $t\left(\beta\right)+\beta\alpha<0$.
This contradicts Lemma \ref{lem:upper-bound} and thus proves the
first assertion in (\ref{enu:ourside-range-empty}). Finally, if $\alpha\notin-\overline{t'\left(\R\right)}$
then we have $-t^{*}\left(-\alpha\right)=-\infty$ by Lemma \ref{lem:conjugate}
(\ref{enu:conjugate-facts-2}). Hence, we have $\mathcal{F}^{\sharp}\left(\alpha,\psi\right)=\emptyset$.
Since $\mathcal{F}\left(\alpha,\psi\right)\subset\mathcal{F}^{\sharp}\left(\alpha,\psi\right)$,
the proof is complete.
\end{proof}
For an expanding multi-map which satisfies the open set condition,\foreignlanguage{english}{
we prove the following lower bound for the Hausdorff dimension of
$\nu_{\beta}$ }by using estimates from \foreignlanguage{english}{\cite[Section 5]{MR2153926}.
Related results and similar arguments can be found in \cite[Lemma 2]{MR1435198}
for conformal repellers, and in \cite[Theorem 4.4.2]{MR2003772} for
graph directed Markov systems. Recall that, for a Borel probability
measure $\nu$ on $J\left(G\right)$, the Hausdorff dimension of $\nu$
(cf. \cite{falconerfractalgeometryMR2118797}) is given by 
\[
\dim_{H}\left(\nu\right):=\inf\left\{ \dim_{H}\left(A\right):A\subset J\left(G\right)\mbox{ is a Borel set with }\nu\left(A\right)=1\right\} .
\]
 }
\begin{prop}
\label{prop:local-dimension-lowerbound}Let $f=\left(f_{i}\right)_{i\in I}\in\left(\Rat\right)^{I}$
be expanding and let \foreignlanguage{english}{\textup{$G=\left\langle f_{i}:i\in I\right\rangle $.}}
Suppose that $f$ satisfies the open set condition. Let $\psi=\left(\psi_{i}\right)_{i\in I}$
be a H\"older family associated with $f$ and let $t:\R\rightarrow\R$
denote the free energy function for $\left(f,\psi\right)$. For each
$\beta\in\R$ we have that 
\[
\dim_{H}\left(\nu_{\beta}\right)\ge-t^{*}\left(t'\left(\beta\right)\right).
\]
In particular, we have $\dim_{H}\left(\mathcal{F}\left(\alpha,\psi\right)\right)\ge-t^{*}\left(-\alpha\right)$
for each $\alpha\in-t'\left(\R\right)$. \end{prop}
\begin{proof}
We  use some estimates and notations from \foreignlanguage{english}{\cite[Section 5]{MR2153926}}.
Suppose that $f$ satisfies the open set condition with open set $U\subset\Chat$.
We may assume that there exists $\epsilon>0$ such that $B(\overline{U},2\epsilon)\cap P(G)=\emptyset$.
Let $\overline{U}=\bigcup_{j=1}^{k}K_{j}$ be a measurable partition
such that $\Int(K_{j})\neq\emptyset$ and $\diam(K_{j})\le\epsilon/10$
for each $j\in\left\{ 1,\dots,k\right\} $.

Let $\beta\in\R$ and $\alpha=-t'(\beta)$. Our main task is to prove
that there exists a constant $C>0$ with the property that, for each
$\Delta>0$ there exist $r_{0}\left(\Delta\right)>0$ and a Borel
set $\tilde{E}\left(\alpha,\Delta\right)\subset J\left(\tilde{f}\right)$
with $\tilde{v}_{\beta}\left(\tilde{E}\left(\alpha,\Delta\right)\right)>0$,
such that for all $r\le r_{0}\left(\Delta\right)$ and $z\in J(G)$
we have 
\begin{equation}
\tilde{v}_{\beta}\left(\pi_{\Chat}^{-1}\left(B(z,r)\right)\cap\tilde{E}\left(\alpha,\Delta\right)\right)\le Cr^{t(\beta)+\beta\alpha-\Delta}.\label{eq:localdimension-lowerbound-1}
\end{equation}
For the Borel probability measure \foreignlanguage{english}{$\nu_{\beta,\Delta}$
on $J\left(G\right)$, given by} $\nu_{\beta,\Delta}\left(A\right):=\tilde{v}_{\beta}\left(\pi_{\Chat}^{-1}\left(A\right)\cap\tilde{E}\left(\alpha,\Delta\right)\right)/\tilde{\nu}_{\beta}\left(\tilde{E}\left(\alpha,\Delta\right)\right)$,
for $A\subset J\left(G\right)$, we then have for each $z\in J\left(G\right)$,
\[
\liminf_{r\rightarrow0}\frac{\log\nu_{\beta,\Delta}\left(B(z,r)\right)}{\log r}=\liminf_{r\rightarrow0}\frac{\log\tilde{v}_{\beta}\left(\pi_{\Chat}^{-1}\left(B(z,r)\right)\cap\tilde{E}\left(\alpha,\Delta\right)\right)}{\log r}\ge t\left(\beta\right)+\beta\alpha-\Delta.
\]
Hence, we have $\dim_{H}\left(\nu_{\beta,\Delta}\right)\ge t\left(\beta\right)+\beta\alpha-\Delta$
by \cite{MR684248}, which gives $\dim_{H}\left(\nu_{\beta}\right)\ge\dim_{H}\left(\nu_{\beta,\Delta}\right)\ge t\left(\beta\right)+\beta\alpha-\Delta$.
Letting $\Delta$ tend to zero, gives that $\dim_{H}\left(\nu_{\beta}\right)\ge t\left(\beta\right)+\beta\alpha$.
Finally, since there exists a Borel subset $A$ of $J\left(G\right)$
with $A\subset\mathcal{F}\left(\alpha,\psi\right)$ and $\nu_{\beta}\left(A\right)=1$
by Lemma \ref{lem:support-of-nu-beta} and Remark \ref{borelsest-of-fullmeasure},
we have $\dim_{H}\left(\mathcal{F}\left(\alpha,\psi\right)\right)\ge\dim_{H}\left(\nu_{\beta}\right)\ge t\left(\beta\right)+\beta\alpha$,
which finishes the proof. 

To prove (\ref{eq:localdimension-lowerbound-1}), let $z\in J(G)$
and let $\Delta>0$. By Lemma \ref{lem:support-of-nu-beta} we have
$\tilde{\nu}_{\beta}\left(\mathcal{\tilde{F}}\left(\alpha,\psi\right)\right)=1$.
By Egoroff's Theorem, there exist a Borel set $\tilde{E}\left(\alpha,\Delta\right)\subset\tilde{\mathcal{F}}\left(\alpha,\psi\right)$
with \foreignlanguage{english}{$\tilde{\nu}_{\beta}\left(\tilde{E}\left(\alpha,\Delta\right)\right)>0$
and  $N\left(\Delta\right)\in\N$ such that 
\begin{equation}
\inf_{y\in\tilde{E}\left(\alpha,\Delta\right)}\beta\frac{S_{n}\tilde{\psi}\left(y\right)}{S_{n}\tilde{\zeta}\left(y\right)}\ge\beta\alpha-\Delta,\mbox{ for all }n\ge N\left(\Delta\right).\label{eq:ergodic-ratio-uniformbound}
\end{equation}
With the notation from \cite[Lemma 5.15]{MR2153926}, we have for
each $r>0$, 
\begin{equation}
\pi_{\Chat}^{-1}\left(B(z,r)\right)\cap J\left(\tilde{f}\right)\subset\bigcup_{i=1}^{p}\eta^{i}\left(\pi_{\Chat}^{-1}\left(B\left(K_{v_{i}},\epsilon/5\right)\right)\cap J\left(\tilde{f}\right)\right),\label{eq:good-cover-sumi}
\end{equation}
where $p\in\N$, $1\le v_{i}\le k$ and $\eta^{i}\left(\kappa,y\right)=\bigl(\omega^{i}\kappa,\gamma_{1}^{i}\dots\gamma_{l_{i}}^{i}\left(y\right)\bigr)$,
$\kappa\in I^{\N}$, $y\in B\left(K_{v_{i}},\epsilon/5\right)$, for
each $1\le i\le p$. Here, we have $\omega^{i}\in I^{l_{i}}$, $l_{i}\in\N$
and $\gamma_{1}^{i}\dots\gamma_{l_{i}}^{i}$ is an inverse branch
of $f_{\omega^{i}}$ in a neighborhood of $B\left(K_{v_{i}},\epsilon/5\right)$
and $\gamma_{j}^{i}$ is an inverse branch of $f_{\omega_{j}^{i}}$,
for every $1\le i\le p$. It is important to note that $p\le C_{4}$,
for some constant $C_{4}$ independent of $r$ and $z$ by \cite[(12)]{MR2153926}.
Moreover, by \cite[(13)]{MR2153926}, we have that 
\begin{equation}
\big\Vert(\gamma_{1}^{i}\circ\dots\circ\gamma_{l_{i}}^{i})'\left(y\right)\big\Vert\le C_{5}r,\quad y\in B(K_{v_{i}},\epsilon/5),\label{eq:derivative-bound}
\end{equation}
with some constant $C_{5}$ independent of $r$ and $y$. In particular,
we have that $l_{i}$ tends to infinity as $r$ tends to zero. Hence,
there exists $r_{0}\left(\Delta\right)>0$ such that, for each $z$
and $r\le r_{0}\left(\Delta\right)$, we have $l_{i}\ge N\left(\Delta\right)$
for $1\le i\le p$ in (\ref{eq:good-cover-sumi}). Then we obtain
by (\ref{eq:good-cover-sumi}) and }$\beta\tilde{\psi}+t\left(\beta\right)\tilde{\zeta}$-conformal\foreignlanguage{english}{ity
of }$\tilde{v}_{\beta}$ \foreignlanguage{english}{that 
\begin{align}
\tilde{v}_{\beta}\left(\pi_{\Chat}^{-1}\left(B(z,r)\right)\cap\tilde{E}\left(\alpha,\Delta\right)\right) & \le\sum_{i=1}^{p}\tilde{v}_{\beta}\left(\eta^{i}\left(\pi_{\Chat}^{-1}\left(B\left(K_{v_{i}},\epsilon/5\right)\right)\right)\cap\tilde{E}\left(\alpha,\Delta\right)\right)\nonumber \\
 & =\sum_{i=1}^{p}\int_{\pi_{\Chat}^{-1}\left(B\left(K_{v_{i}},\epsilon/5\right)\right)}\e^{S_{l_{i}}\left(\beta\tilde{\psi}+t\left(\beta\right)\tilde{\zeta}\right)\left(\eta^{i}\right)}\1_{\tilde{E}\left(\alpha,\Delta\right)}\left(\eta^{i}\right)d\tilde{\nu}_{\beta}.\label{eq:conformality-estimate}
\end{align}
The estimate in (\ref{eq:ergodic-ratio-uniformbound}) gives that
\begin{equation}
\1_{\tilde{E}\left(\alpha,\Delta\right)}S_{l_{i}}\left(t\left(\beta\right)\tilde{\zeta}+\beta\tilde{\psi}\right)=\1_{\tilde{E}\left(\alpha,\Delta\right)}S_{l_{i}}\tilde{\zeta}\Bigl(t(\beta)+\beta S_{l_{i}}\tilde{\psi}\big/S_{l_{i}}\tilde{\zeta}\Bigr)\le\1_{\tilde{E}\left(\alpha,\Delta\right)}S_{l_{i}}\tilde{\zeta}\left(t(\beta)+\beta\alpha-\Delta\right).\label{eq:good-indicator}
\end{equation}
Combining the estimates in (\ref{eq:derivative-bound}), (\ref{eq:conformality-estimate})
and (\ref{eq:good-indicator}), we obtain 
\[
\tilde{v}_{\beta}\left(\pi_{\Chat}^{-1}\left(B(z,r)\right)\cap\tilde{E}\left(\alpha,\Delta\right)\right)\le C_{4}(C_{5}r)^{t(\beta)+\beta\alpha-\Delta},
\]
which completes  the proof.}
\end{proof}
\selectlanguage{english}%
We can now state the main result of this section, which establishes
the multifractal formalism for H\"older families associated with
expanding rational semigroups.
\selectlanguage{british}%
\begin{thm}
\label{thm:multifractalformalism}Let $f=\left(f_{i}\right)_{i\in I}\in\left(\Rat\right)^{I}$
be expanding. Let $\psi=\left(\psi_{i}\right)_{i\in I}$ be a H\"older
family associated with $f$ and let $t:\R\rightarrow\R$ denote the
free energy function for $\left(f,\psi\right)$. Suppose there exists
$\gamma\in\R$ such that $\mathcal{P}\bigl(\gamma\tilde{\zeta},\tilde{f}\bigr)=0$.
Let $\alpha_{\pm}:=\alpha_{\pm}\left(\psi\right)$ and $\alpha_{0}:=\alpha_{0}\left(\psi\right)$.
Then we have the following.
\begin{enumerate}
\item \label{enu:range-of-spectrum}If $\alpha_{-}=\alpha_{+}$ then we
have that $\alpha_{-}=\alpha_{0}=\alpha_{+}$, $-t'\left(\R\right)=\left\{ \alpha_{0}\right\} $
and $\mathcal{F}\left(\alpha,\psi\right)$ is non-empty if and only
if $\alpha=\alpha_{0}$. If $\alpha_{-}<\alpha_{+}$ then we have
that $-t'\left(\R\right)=\left(\alpha_{-},\alpha_{+}\right)$, each
$\mathcal{F}\left(\alpha,\psi\right)$ is non-empty for $\alpha\in\left(\alpha_{-},\alpha_{+}\right)$,
$s(\alpha ):=-t^{*}\left(-\alpha \right)$ is a strictly concave
real analytic positive function on $\left(\alpha_{-},\alpha_{+}\right)$
with maximum value $\delta=-t^{*}\left(-\alpha_{0}\right)>0$, and 
$s''<0$ on $(\alpha _{-},\alpha _{+}).$ 
\item \label{enu:upperbound-of-spectrum}For each $\alpha\in\R$ we have
that 
\[
\dim_{H}\left(\mathcal{F}\left(\alpha,\psi\right)\right)\le\dim_{H}\left(\mathcal{F}^{\sharp}\left(\alpha,\psi\right)\right)\le\max\left\{ -t^{*}\left(-\alpha\right),0\right\} .
\]

\item \label{enu:lowerbound}If $f$ satisfies the open set condition, then
for each $\alpha\in-t'\left(\R\right)$ we have that 
\[
\dim_{H}\left(\mathcal{F}\left(\alpha,\psi\right)\right)=\dim_{H}\left(\mathcal{F}^{\sharp}\left(\alpha,\psi\right)\right)=-t^{*}\left(-\alpha\right)>0.
\]
In particular, we have $\dim_{H}\left(\mathcal{F}\left(\alpha_{0},\psi\right)\right)=\delta>0$. 
\end{enumerate}
\end{thm}
\begin{proof}
We start with the proof of (\ref{enu:range-of-spectrum}). We distinguish
two cases. First suppose that $\tilde{\mu}_{0}=\tilde{\mu}_{\gamma}$.
Then we have that $-t'\left(\R\right)=\left\{ -t'\left(0\right)\right\} =\left\{ \alpha_{0}\right\} $
by Propositions \ref{prop:characterisation-strict-convexity} 
and \ref{prop:free-energy-function}. By Lemma \ref{lem:support-of-nu-beta}
we have that $\mathcal{F}\left(\alpha_{0},\psi\right)\neq\emptyset$.
By Lemma \ref{lem:t-star-properties} (\ref{enu:ourside-range-empty}),
we have that $\mathcal{F}\left(\alpha,\psi\right)=\emptyset$ if $\alpha\neq\alpha_{0}$.
Hence, we have that $\alpha_{-}=\alpha_{0}=\alpha_{+}$. Now suppose
that $\tilde{\mu}_{0}\neq\tilde{\mu}_{\gamma}$. Then $t''>0$ on $\R $ 
by Proposition \ref{prop:characterisation-strict-convexity}
 and we have $\mathcal{F}\left(\alpha,\psi\right)\neq\emptyset$
for $\alpha\in-t'\left(\R\right)$ by Lemma \ref{lem:support-of-nu-beta}.
Combining with Lemma \ref{lem:t-star-properties} (\ref{enu:ourside-range-empty}),
we obtain that $-t'\left(\R\right)=\left(\alpha_{-},\alpha_{+}\right)$.
That $s(\alpha ):= -t^{*}\left(-\alpha \right)$ is a strictly concave
real analytic positive function on $\left(\alpha_{-},\alpha_{+}\right)$
with maximum value $\delta=-t^{*}\left(-\alpha_{0}\right)>0$ and $s''<0$ on 
$(\alpha _{-},\alpha _{+})$ follows
from Lemma \ref{lem:legendre-transform-convex-conjugate}. 

To prove (\ref{enu:upperbound-of-spectrum}), let $\alpha\in\R$ and
suppose that $\alpha\ge\alpha_{0}$. The case $\alpha<\alpha_{0}$
can be proved similarly. Since the upper bound in (\ref{enu:upperbound-of-spectrum})
clearly holds if $\mathcal{F}^{\sharp}\left(\alpha,\psi\right)=\emptyset$
we may assume that $\mathcal{F}^{\sharp}\left(\alpha,\psi\right)\neq\emptyset$.
Hence, we have $-t^{*}\left(-\alpha\right)\ge0$ by Lemma \ref{lem:t-star-properties}
(\ref{enu:ourside-range-empty}). Let $z\in\mathcal{F}^{\sharp}\left(\alpha,\psi\right)$.
By Lemma \ref{lem:upper-bound} we have for each $\beta\le0$, 
\begin{equation}
0\le\liminf_{r\rightarrow0}\frac{\log\nu_{\beta}\left(B\left(z,r\right)\right)}{\log r}\le t\left(\beta\right)+\beta\alpha.\label{eq:upperbound-proof-1}
\end{equation}
Since $\inf_{\beta\le0}\left\{ t\left(\beta\right)+\beta\alpha\right\} =-t^{*}\left(-\alpha\right)$
by Lemma \ref{lem:t-star-properties} (\ref{enu:t-star-formula}),
it follows from (\ref{eq:upperbound-proof-1}) and \cite[Proposition 4.9 (b)]{falconerfractalgeometryMR2118797}
and its proof that we have $\dim_{H}\left(\mathcal{F}^{\sharp}\left(\alpha,\psi\right)\right)\le-t^{*}\left(-\alpha\right)$.
Since $\mathcal{F}\left(\alpha,\psi\right)\subset\mathcal{F}^{\sharp}\left(\alpha,\psi\right)$,
the proof of (\ref{enu:upperbound-of-spectrum}) is complete.

To prove (\ref{enu:lowerbound}), suppose that $f$ satisfies the
open set condition and let $\alpha=-t'\left(\beta\right)$ for some
$\beta\in\R$. By Proposition \ref{prop:local-dimension-lowerbound}
we have $\dim_{H}\left(\mathcal{F}\left(\alpha,\psi\right)\right)\ge-t^{*}\left(-\alpha\right)$.
By Lemma \ref{lem:legendre-transform-convex-conjugate} we have $-t^{*}\left(-\alpha\right)>0$.
Combining these estimates with the upper bound in (\ref{enu:upperbound-of-spectrum})
completes the proof of the theorem. 
\end{proof}
The next lemma shows that, for a Bernoulli family $\psi$, a trivial
multifractal spectrum occurs in a very special situation. For a compact
metric space $X$, we denote by $C\left(X\right)$ the space of all
complex-valued continuous functions endowed with the supremum norm.
\begin{prop}
\label{lem:degenerate-via-zdunik}Let $f=\left(f_{i}\right)_{i\in I}\in\left(\Rat\right)^{I}$
be expanding and let \foreignlanguage{english}{\textup{$G=\left\langle f_{i}:i\in I\right\rangle $.
}}Suppose that $\deg\left(f_{i_{0}}\right)\ge2$ for some $i_{0}\in I$.
Let $\left(c_{i}\right)_{i\in I}$ be a family of negative numbers
and let $\psi=\left(\psi_{i}:f_{i}^{-1}\left(J\left(G\right)\right)\rightarrow\R\right)_{i\in I}$
be given by $\psi_{i}\left(z\right)=c_{i}$ for each $i\in I$ and
$z\in f_{i}^{-1}\left(J\left(G\right)\right)$. Let $t:\R\rightarrow\R$
denote the free energy function for $\left(f,\psi\right)$. Let $\gamma$
be the unique number such that $\mathcal{P}\left(\gamma\tilde{\psi},\tilde{f}\right)=0$.
Then we have $\alpha_{-}\left(\psi\right)=\alpha_{+}\left(\psi\right)$
if and only if there exist an automorphism $\varphi\in\Aut\bigl(\Chat\bigr)$,
complex numbers $\left(a_{i}\right)_{i\in I}$ and $\lambda\in\R$
such that for all $i\in I$, 
\begin{equation}
\varphi\circ f_{i}\circ\varphi^{-1}\left(z\right)=a_{i}z^{\pm\deg\left(f_{i}\right)}\quad\mbox{and}\quad\log\deg\left(f_{i}\right)=\lambda c_{i}.\label{eq:degenerate-conditions}
\end{equation}
Moreover, if the assertions in (\ref{eq:degenerate-conditions}) hold,
then we have $\lambda=-\left(\gamma/\delta\right)$. \end{prop}
\begin{proof}
First note that we have $\alpha_{-}\left(\psi\right)=\alpha_{+}\left(\psi\right)$
if and only if $\tilde{\mu}_{0}=\tilde{\mu}_{\gamma}$ by Theorem
\ref{thm:multifractalformalism} (\ref{enu:range-of-spectrum}) and
Proposition \ref{prop:characterisation-strict-convexity}. Now suppose
that $\alpha_{-}\left(\psi\right)=\alpha_{+}\left(\psi\right)$. 
Then $\tilde{\mu }_{0}=\tilde{\mu }_{\gamma }.$  By  Proposition \ref{prop:characterisation-strict-convexity} there exists a continuous function $v:J(\tilde{f})\rightarrow \R $ such that
$\delta \tilde{\zeta }=\gamma \tilde{\psi }+v-v\circ \tilde{f}.$

For each $n\in\N$, $\xi\in I^{n}$ and $u\in\R$, we denote by $p\left(u,\xi\right)$
the topological pressure of the potential $u\log\Vert f_{\xi}'\Vert:J_{\overline{\xi}}\rightarrow\R$
with respect to $f_{\xi}$, where $\overline{\xi}:=\left(\xi_{1},\dots,\xi_{n},\xi_{1},\dots,\xi_{n},\dots\right)\in I^{\N}$.
Note that $J_{\overline{\xi}}=J\left(f_{\xi}\right)$. Our next aim
is to show that the function $u\mapsto p\left(u,\xi\right)$, $u\in\R$,
is constant. By \cite[Theorem 5.6.5]{MR2656475} we have that, for
each $u\in\R$ there exists an $f_{\xi}$-invariant Borel probability
measure $m$ on $J_{\overline{\xi}}$ such that 
\begin{equation}
\frac{\partial}{\partial u}p\left(u,\xi\right)=\int_{J_{\overline{\xi}}}\log\Vert f_{\xi}'\Vert dm.\label{eq:degenerate-criterion-1}
\end{equation}
Denote by $\tilde{m}$ the Borel probability measure supported on
$J^{\overline{\xi}}$ which is given by $\tilde{m}\left(\bigl\{\overline{\xi}\bigr\}\times A\right):=m\left(A\right)$,
for each Borel set $A\subset J_{\overline{\xi}}$. Then $\tilde{m}$
is $\tilde{f}^{n}\big|_{J^{\overline{\xi}}}$-invariant. From this
and (\ref{eq:cohomologous-potentials}) we deduce that 
\begin{equation}
\frac{\partial}{\partial u}p\left(u,\xi\right)=-\int_{J^{\overline{\xi}}}S_{n}\tilde{\zeta}d\tilde{m}=-\left(\gamma/\delta\right)\int_{J^{\overline{\xi}}}S_{n}\tilde{\psi}d\tilde{m}=-\left(\gamma/\delta\right)\sum_{i=1}^{n}c_{\xi_{i}}.\label{eq:degenerate-criterion-2}
\end{equation}
Hence, the function $u\mapsto p\left(u,\xi\right)$, $u\in\R$, is
constant. Now, similarly as in \cite[Proof of Theorem 3.1]{su09},
using Zdunik's theorem (\cite{MR1032883}), we obtain that there exist
an automorphism $\varphi\in\Aut\bigl(\Chat\bigr)$ and complex numbers
$\left(a_{i}\right)_{i\in I}$ such that 
\[
\varphi\circ f_{i}\circ\varphi^{-1}\left(z\right)=a_{i}z^{\pm\deg\left(f_{i}\right)},\quad z\in\Chat.
\]
Since $\deg\left(f_{i_{0}}\right)\ge2$ and $f$ is expanding, it
follows that $\deg\left(f_{i}\right)\ge2$ for all $i\in I$. Moreover,
by combining (\ref{eq:degenerate-criterion-1}) and (\ref{eq:degenerate-criterion-2}),
we have that $\log\deg\left(f_{i}\right)=-\left(\gamma/\delta\right)c_{i}$
for each $i\in I$. 

To prove the converse implication, suppose that there exist an automorphism
$\varphi\in\Aut\bigl(\Chat\bigr)$, complex numbers $\left(a_{i}\right)_{i\in I}$
and $\lambda\in\R$, such that $\varphi\circ f_{i}\circ\varphi^{-1}\left(z\right)=a_{i}z^{\pm\deg\left(f_{i}\right)}$
and $\log\deg\left(f_{i}\right)=\lambda c_{i}$, for each $i\in I$.
It follows that, for each $n\in\N$, $\xi\in I^{n}$ and $z\in J_{\overline{\xi}}$
 such that $f_{\xi}\left(z\right)=z$, we have that 
\[
\log\Vert f_{\xi}'\left(z\right)\Vert=\sum_{i=1}^{n}\log\deg\left(f_{\xi_{i}}\right)=\lambda\sum_{i=1}^{n}c_{\xi_{i}}.
\]
Hence, we have $-S_{n}\tilde{\zeta}\bigl(\overline{\xi},z\bigr)=\lambda S_{n}\tilde{\psi}\bigl(\overline{\xi},z\bigr)$.
Since $\tilde{f}:J\left(\tilde{f}\right)\rightarrow J\left(\tilde{f}\right)$
is an open, distance expanding, topologically transitive map (see
e.g. Lemma \ref{lem:ratsemi-facts-basic} (\ref{enu:preimages-dense})
for the transitivity), it follows from a Livsic type theorem that
there exists a continuous function $\tilde{h}:J\left(\tilde{f}\right)\rightarrow\R$
such that $-\tilde{\zeta}=\lambda\tilde{\psi}+\tilde{h}-\tilde{h}\circ\tilde{f}$
(see e.g. \cite[Proposition 4.4.5]{MR2656475}). In particular, we
have $-\delta\tilde{\zeta}=\delta\lambda\tilde{\psi}+\delta\left(\tilde{h}-\tilde{h}\circ\tilde{f}\right)$,
which shows that $\lambda=-\left(\gamma/\delta\right)$. Thus, the
potentials $\delta\tilde{\zeta}$ and $\gamma\tilde{\psi}$ have the
same equilibrium state, which means  that $\tilde{\mu}_{0}=\tilde{\mu}_{\gamma}$.
The proof is complete.
\end{proof}

\section{\label{sec:Application-to-Random}Application to Random Complex Dynamics}

The first lemma relates the H\"older exponent of a function to $\M_{*}$
(see Definition \ref{def:QRH}).
\begin{lem}
\label{lem:hoelderexponent-via-Q}Let $U\subset\Chat$ be an open
set and let $\rho:U\rightarrow\C$ be a bounded function.\foreignlanguage{english}{\textup{
Then we have for each $z\in U$, 
\[
\Hol\left(\rho,z\right)=\M_{*}\left(\rho,z\right).
\]
}}\end{lem}
\begin{proof}
Let $\beta>\Hol\left(\rho,z\right)$. Then we have $\limsup_{y\rightarrow z,y\neq z}\left|\rho\left(y\right)-\rho\left(z\right)\right|\big/d\left(y,z\right)^{\beta}=\infty$,
which implies that \linebreak  $\limsup_{y\rightarrow z,y\neq z}\log\left|\rho\left(y\right)-\rho\left(z\right)\right|-\beta\log d\left(y,z\right)=\infty$.
Hence, we have that 
\[
\limsup_{y\rightarrow z,y\neq z}\left(-\log d\left(y,z\right)\right)\left(\frac{\log\left|\rho\left(y\right)-\rho\left(z\right)\right|}{-\log d\left(y,z\right)}+\beta\right)=\infty.
\]
Since $\lim_{y\rightarrow z,y\neq z}\left(-\log d\left(y,z\right)\right)=\infty$,
we conclude that $\limsup_{y\rightarrow z,y\neq z}\log\left|\rho\left(y\right)-\rho\left(z\right)\right|\big/\left(-\log d\left(y,z\right)\right)+\beta\ge0$,
which implies $\liminf_{y\rightarrow z,y\neq z}\log\left|\rho\left(y\right)-\rho\left(z\right)\right|\big/\log d\left(y,z\right)\le\beta$.
We have thus shown that $\M_{*}\left(\rho,z\right)\le\Hol\left(\rho,z\right)$. 

Let $\beta<\Hol\left(\rho,z\right)$. Then we have $\lim_{y\rightarrow z,y\neq z}\left|\rho\left(y\right)-\rho\left(z\right)\right|\big/d\left(y,z\right)^{\beta}=0$,
which implies that \linebreak  $\lim_{y\rightarrow z,y\neq z}\log\left|\rho\left(y\right)-\rho\left(z\right)\right|-\beta\log d\left(y,z\right)=-\infty$.
Hence, we have that 
\[
\lim_{y\rightarrow z,y\neq z}\left(-\log d\left(y,z\right)\right)\left(\frac{\log\left|\rho\left(y\right)-\rho\left(z\right)\right|}{-\log d\left(y,z\right)}+\beta\right)=-\infty.
\]
We conclude that $\limsup_{y\rightarrow z,y\neq z}\log\left|\rho\left(y\right)-\rho\left(z\right)\right|\big/\left(-\log d\left(y,z\right)\right)+\beta\le0$,
which then implies that \linebreak  $\liminf_{y\rightarrow z,y\neq z}\log\left|\rho\left(y\right)-\rho\left(z\right)\right|\big/\log d\left(y,z\right)\ge\beta$.
We have thus shown that $\M_{*}\left(\rho,z\right)\ge\Hol\left(\rho,z\right)$
and the proof of the lemma is complete.
\end{proof}
The following lemma allows us to investigate the H\"older exponent
of a non-constant unitary eigenfunction of $M_{\tau}$ by means of
ergodic sums with respect to the skew product associated with a rational
semigroup. 
\begin{lem}
\label{lem:Q-via-ergodicsums}Let $f=\left(f_{i}\right)_{i\in I}\in\left(\Rat\right)^{I}$
be expanding and let \foreignlanguage{english}{\textup{$G=\left\langle f_{i}:i\in I\right\rangle $.}}
Suppose that $f$ satisfies the separation condition. Let $\left(p_{i}\right)_{i\in I}\in\left(0,1\right)^{I}$
be a probability vector, let $\tau:=\sum_{i\in I}p_{i}\delta_{f_{i}}$
and let $\rho\in C\bigl(\Chat\bigr)$ be a non-constant function belonging
to $U_{\tau}$. Let $\psi=\left(\psi_{i}:f_{i}^{-1}\left(J\left(G\right)\right)\rightarrow\R\right)_{i\in I}$
be given by $\psi_{i}\left(z\right):=\log p_{i}$ for each $i\in I$.
Then \foreignlanguage{english}{for each $\left(\omega,z\right)\in J\left(\tilde{f}\right)$
we have that 
\[
\liminf_{n\rightarrow\infty}\frac{S_{n}\tilde{\psi}\left(\left(\omega,z\right)\right)}{S_{n}\tilde{\zeta}\left(\left(\omega,z\right)\right)}=\M_{*}\left(\rho,z\right)\quad\mbox{and}\quad\limsup\frac{S_{n}\tilde{\psi}\left(\left(\omega,z\right)\right)}{S_{n}\tilde{\zeta}\left(\left(\omega,z\right)\right)}=\M^{*}\left(\rho,z\right).
\]
}\end{lem}
\begin{proof}
We proceed similarly as in the proof of \foreignlanguage{english}{\cite[Lemma 5.48]{s11random}.
By \cite[Theorem 3.15 (10)]{s11random} we may assume that $M_{\tau}\left(\rho\right)=\rho$.
Since $f$ satisfies the separation condition, }we conclude that there
exists $r_{0}>0$ such that, for all $i,j\in I$ with $i\neq j$ and
$y\in f_{i}^{-1}\left(J\left(G\right)\right)$, we have $f_{j}\left(B\left(y,r_{0}\right)\right)\subset F\left(G\right)$. 

\selectlanguage{english}%
Let $\left(\omega,z\right)\in J\left(\tilde{f}\right)$. \foreignlanguage{british}{Since
$f$ is expanding, we have that $G$ is hyperbolic and there exists
a non-empty $G$-forward invariant compact subset of $F\left(G\right)$
by Proposition \ref{prop:expandingness-criteria} (\ref{enu:exp-implies-hyperbolicloxodromic}).
Hence, there exists $R>0$ such that, for each $n\in\N$, }there exists
a holomorphic branch $\phi_{n}:B\bigl(f_{\omega_{|n}}\left(z\right),R\bigr)\rightarrow\Chat$
of $f_{\omega_{|n}}^{-1}$ such that $f_{\omega_{|n}}\left(\phi_{n}\left(y\right)\right)=y$
for $y\in B\bigl(f_{\omega_{|n}}\left(z\right),R\bigr)$ and $\phi_{n}\bigl(f_{\omega_{|n}}\left(z\right)\bigr)=z$.
After making $r_{0}$ sufficiently small, we may assume that, for
the sets $B_{n}$, which are for $n\in\N$ given by 
\[
B_{n}:=\phi_{n}\bigl(B\bigl(f_{\omega_{|n}}\left(z\right),r_{0}\bigr)\bigr),
\]
we have that $\diam\bigl(f_{\omega_{|k}}\left(B_{n}\right)\bigr)\le r_{0}$
for all $1\le k\le n$. Combining this with our assumption that $M_{\tau}\left(\rho\right)=\rho$
and that $\rho$ is constant on each connected component of $F\left(G\right)$
by \cite[Theorem 3.15 (1)]{s11random}, we obtain that for all $a,b\in B_{n}$,
\begin{equation}
\left|\rho\left(a\right)-\rho\left(b\right)\right|=p_{\omega_{1}}\cdot\dots\cdot p_{\omega_{n}}\left|\rho\bigl(f_{\omega_{|n}}\left(a\right)\bigr)-\rho\bigl(f_{\omega_{|n}}\left(b\right)\bigr)\right|.\label{eq:functional-equation}
\end{equation}
We set $r_{n}:=\big\Vert f_{\omega_{|n}}'\left(z\right)\big\Vert^{-1}$
for each $n\in\N$. W\foreignlanguage{british}{e may assume that $\left(r_{n}\right)_{n\in\N}$
is strictly decreasing} because $f$ is\foreignlanguage{british}{
expanding. Hence, for each $r>0$ sufficiently small, there exists
a unique $n\in\N$ such that $r_{n+1}\le r\le r_{n}$.} To prove the
lemma, our main task is to verify that there exists a constant $C>0$,
such that for all $r>0$ sufficiently small, 
\begin{equation}
C^{-1}\le\frac{\sup\left\{ \left|\rho\left(z\right)-\rho\left(y\right)\right|:y\in B\left(z,r\right)\right\} }{p_{\omega_{1}}\cdot\dots\cdot p_{\omega_{n}}}\le C.\label{eq:limit-function-versus-ergodicsums}
\end{equation}
To prove \foreignlanguage{british}{(\ref{eq:limit-function-versus-ergodicsums})},
we first observe that by Koebe's distortion theorem, there exist positive
constants $c_{1},c_{2}$ such that for each $n\in\N$, 
\[
B\left(z,c_{1}r_{n}\right)\subset B_{n}\subset B\left(z,c_{2}r_{n}\right).
\]
Moreover, it is shown in \cite{s11random} that 
\[
J\left(G\right)=\left\{ y\in\Chat:\forall\epsilon>0:\rho_{|B\left(y,\epsilon\right)}\mbox{ is not constant}\right\} .
\]
Since $J\left(G\right)$ is compact, we see that there exists $C_{1}>0$
such that for all $y\in J\left(G\right)$, 
\begin{equation}
\sup\left\{ \left|\rho\left(y\right)-\rho\left(y'\right)\right|:y'\in B\left(y,r_{0}\right)\right\} \ge C_{1}.\label{eq:non-constant-estimate}
\end{equation}
Further, there exists $k\in\N$ such that $c_{2}r_{n+k}\le r_{n+1}$
and $r_{n}\le c_{1}r_{n-k}$, for all $n$ sufficiently large. Consequently,
for each $r>0$ we have 
\[
B\left(z,r\right)\supset B\left(z,r_{n+1}\right)\supset B\left(z,c_{2}r_{n+k}\right)\supset B_{n+k}\quad\mbox{and}\quad B\left(z,r\right)\subset B\left(z,r_{n}\right)\subset B\left(z,c_{1}r_{n-k}\right)\subset B_{n-k}.
\]
Combining with (\ref{eq:functional-equation}) and (\ref{eq:non-constant-estimate})
we obtain that 
\begin{align*}
\sup\left\{ \left|\rho\left(z\right)-\rho\left(y\right)\right|:y\in B\left(z,r\right)\right\}  & \ge\sup\left\{ \left|\rho\left(z\right)-\rho\left(y\right)\right|:y\in B_{n+k}\right\} \\
 & =\sup\left\{ p_{\omega_{1}}\cdot\dots\cdot p_{\omega_{n+k}}\left|\rho\bigl(f_{\omega_{|n+k}}\left(z\right)\bigr)-\rho\bigl(f_{\omega_{|n+k}}\left(y\right)\bigr)\right|:y\in B_{n+k}\right\} \\
 & \ge p_{\omega_{1}}\cdot\dots\cdot p_{\omega_{n+k}}C_{1}.
\end{align*}
Similarly, we have 
\begin{align*}
\sup\left\{ \left|\rho\left(z\right)-\rho\left(y\right)\right|:y\in B\left(z,r\right)\right\}  & \le\sup\left\{ \left|\rho\left(z\right)-\rho\left(y\right)\right|:y\in B_{n-k}\right\} \\
 & =\sup\left\{ p_{\omega_{1}}\cdot\dots\cdot p_{\omega_{n-k}}\left|\rho\bigl(f_{\omega_{|n-k}}\left(z\right)\bigr)-\rho\bigl(f_{\omega_{|n-k}}\left(y\right)\bigr)\right|:y\in B_{n-k}\right\} \\
 & \le p_{\omega_{1}}\cdot\dots\cdot p_{\omega_{n-k}}2\max_{v\in C\left(\Chat\right)}\left|\rho\left(v\right)\right|.
\end{align*}
We have thus proved \foreignlanguage{british}{(\ref{eq:limit-function-versus-ergodicsums}). }

\selectlanguage{british}%
By (\ref{eq:limit-function-versus-ergodicsums}) and $r_{n+1}\le r\le r_{n}$
we obtain that for each $r>0$, 
\[
\frac{S_{n}\tilde{\psi}\left(\left(\omega,z\right)\right)+\log C}{S_{n+1}\tilde{\zeta}\left(\left(\omega,z\right)\right)}\le\frac{\log\sup\left\{ \left|\rho\left(z\right)-\rho\left(y\right)\right|:y\in B\left(z,r\right)\right\} }{\log r}\le\frac{S_{n}\tilde{\psi}\left(\left(\omega,z\right)\right)-\log C}{S_{n}\tilde{\zeta}\left(\left(\omega,z\right)\right)}.
\]
The lemma follows by letting $r$ tend to zero.
\end{proof}
We are now in the position to state the main result of this section.
\begin{thm}
\label{thm:mf-for-hoelderexponent}Let $f=\left(f_{i}\right)_{i\in I}\in\left(\Rat\right)^{I}$
be expanding and let \foreignlanguage{english}{\textup{$G=\left\langle f_{i}:i\in I\right\rangle $.}}
Suppose that $f$ satisfies the separation condition. Let $\left(p_{i}\right)_{i\in I}\in\left(0,1\right)^{I}$
be a probability vector, let $\tau:=\sum_{i\in I}p_{i}\delta_{f_{i}}$
and suppose that there exists a non-constant function belonging to
$U_{\tau}$. Let $\rho\in C(\Chat)$ be a non-constant function belonging
to $U_{\tau}$. Let $\psi=\left(\psi_{i}:f_{i}^{-1}\left(J\left(G\right)\right)\rightarrow\R\right)_{i\in I}$
be given by $\psi_{i}\left(z\right):=\log p_{i}$ and let $t:\R\rightarrow\R$
denote the free energy function for $\left(f,\psi\right)$. Let $\gamma$
be the unique number such that $\mathcal{P}\left(\gamma\tilde{\psi},\tilde{f}\right)=0$.
Then we have the following.
\begin{enumerate}
\item There exists a number $a\in\left(0,1\right)$ such that $\rho:\Chat\rightarrow\C$
is $a$-H\"older continuous and $a\le\alpha_{-}\left(\psi\right)$.
\selectlanguage{english}%
\item \textup{\emph{We have $\alpha_{+}(\psi )=\sup\left\{ \alpha\in\R:H\left(\rho,\alpha\right)\neq\emptyset\right\} $
and $\alpha_{-}(\psi )=\inf\left\{ \alpha\in\R:H\left(\rho,\alpha\right)\neq\emptyset\right\} $.
Moreover, $H$ can be replaced by $R_{*},R$ or $R^{*}$. }}
\selectlanguage{british}%
\item Let $\alpha_{\pm}:=\alpha_{\pm}\left(\psi\right)$ and $\alpha_{0}:=\alpha_{0}\left(\psi\right)$.
If $\alpha_{-}<\alpha_{+}$ then we have for each $\alpha\in\left(\alpha_{-},\alpha_{+}\right)$,
\begin{align*}
\dim_{H}\left(\mathcal{F}\left(\alpha,\psi\right)\right) & =\dim_{H}\left(\mathcal{F}^{\sharp}\left(\alpha,\psi\right)\right)=\dim_{H}\left(R^{*}\left(\rho,\alpha\right)\right)=\dim_{H}\left(R_{*}\left(\rho,\alpha\right)\right)\\
 & =\dim_{H}\left(R\left(\rho,\alpha\right)\right)=\dim_{H}\left(H\left(\rho,\alpha\right)\right)=-t^{*}\left(-\alpha\right)>0.
\end{align*}
Moreover, $s(\alpha ):=-t^{*}\left(-\alpha \right)$ is a real analytic
strictly concave positive function on $\left(\alpha_{-},\alpha_{+}\right)$
with maximum value $\delta=-t^{*}\left(-\alpha_{0}\right)>0$.  
Also, $s''<0$ on $(\alpha _{-}, \alpha _{+}).$ 
\item 

\begin{enumerate}
\item For each $i\in I$ we have $\deg\left(f_{i}\right)\ge2$. Moreover,
we have $\alpha_{-}=\alpha_{+}$ if and only if there exist an automorphism
$\varphi\in\Aut\bigl(\Chat\bigr)$ and $\left(a_{i}\right)\in\C^{I}$
such that 
\[
\varphi\circ f_{i}\circ\varphi^{-1}\left(z\right)=a_{i}z^{\pm\deg\left(f_{i}\right)}\quad\mbox{and}\quad\log\deg\left(f_{i}\right)=-\left(\gamma/\delta\right)\log p_{i}.
\]

\item If $\alpha_{-}=\alpha_{+}$ then we have 
\[
\mathcal{F}\left(\alpha_{0},\psi\right) =\mathcal{F}^{\sharp}\left(\alpha_{0},\psi\right)=R^{*}\left(\rho,\alpha_{0}\right)=R_{*}\left(\rho,\alpha_{0}\right)=
 R\left(\rho,\alpha_{0}\right)=H\left(\rho,\alpha_{0}\right)=J(G),
\]
where $\dim_H(J(G))=\delta >0$,  and for all $\alpha\neq\alpha_{0}$  we have 
\[
\mathcal{F}\left(\alpha,\psi\right)=\mathcal{F}^{\sharp}\left(\alpha,\psi\right)=R^{*}\left(\rho,\alpha\right)=R_{*}\left(\rho,\alpha\right)=R\left(\rho,\alpha\right)=H\left(\rho,\alpha\right)=\emptyset.
\]

\end{enumerate}
\end{enumerate}
\end{thm}
\begin{proof}
By the separation condition and $J\left(G\right)=\bigcup_{i\in I}f_{i}^{-1}\left(J\left(G\right)\right)$
(see \cite{MR1767945}) we have that the kernel Julia set $J_{\textrm{ker}}(G):=\bigcap_{g\in G}g^{-1}\left(J\left(G\right)\right)$
of $G$ is empty. From this and the assumption that there exists a
non-constant unitary eigenfunction of $M_{\tau}$ in $C(\Chat)$ (\cite[Theorem 3.15 (21), Remark 3.18]{s11random}),
it follows that $\deg\left(f_{i}\right)\ge2$ for each $i\in I$.
Moreover, by \cite[Theorem 3.29]{S13Coop} there exists a constant
$a\in\left(0,1\right)$ such that $\rho:\Chat\rightarrow\C$ is $a$-H\"older
continuous. 

Since $f$ satisfies the separation condition, by passing to 
$(f_{\omega })_{\omega \in I^{k}}$ where $k$ is a sufficiently large positive integer, 
we may assume that $f$ satisfies
the open set condition. Also note that there exists $\gamma\in\R$
such that $\mathcal{P}\left(\gamma\tilde{\psi},\tilde{f}\right)=0$. 

Let $\alpha\in\R$. Recall that $H\left(\rho,\alpha\right)=R_{*}\left(\rho,\alpha\right)$
by Lemma \ref{lem:hoelderexponent-via-Q}. We only give the proof
of (2), (3) and (4) for the level set $R_{*}\left(\rho,\alpha\right)$.
The sets $R^{*}\left(\rho,\alpha\right)$ and $R\left(\rho,\alpha\right)$
can be considered in a similar fashion. The main task is to show that
$\mathcal{F}\left(\alpha,\psi\right)\subset R_{*}\left(\rho,\alpha\right)\subset\mathcal{F}^{\sharp}\left(\alpha,\psi\right)$.
Then the assertion in (2) follows from Theorem \ref{thm:multifractalformalism}
(1) and Lemma~\ref{lem:t-star-properties} (\ref{enu:ourside-range-empty}), and the assertions in (3) follow from Theorem \ref{thm:multifractalformalism}
(1), (2) and (3). The assertion in (4a) follows from Proposition \ref{lem:degenerate-via-zdunik}. To prove (4b) we observe that by 
the proof of Proposition~\ref{lem:degenerate-via-zdunik}
there exists a continuous function $\tilde{h}:J(\tilde{f})\rightarrow \R$  such that 
\[
\frac{S_n\tilde{\psi}(x)}{S_n\tilde{\zeta}(x)}=\frac{\delta S_n\tilde{\psi}(x)}{\gamma S_n\tilde{\psi}(x)+\tilde{h}(x)-\tilde{h}\circ \tilde{f}^{n+1}(x)},\quad \textrm{for every } x\in J(\tilde{f}),
\]
which shows that $\lim_{n \rightarrow \infty}{S_n\tilde{\psi}(x)}/{S_n\tilde{\zeta}(x)}=\delta / \gamma$ for every $ x\in J(\tilde{f})$. Now (4b) follows from Theorem \ref{thm:multifractalformalism} (1)
and (3) and Lemma~\ref{lem:t-star-properties} (\ref{enu:ourside-range-empty}). Finally, by combining with the fact that $\rho$ is $a$-H\"older
continuous, we obtain that $a\le\alpha_{-}\left(\psi\right)$. 

To complete the proof, we verify that $\mathcal{F}\left(\alpha,\psi\right)\subset R_{*}\left(\rho,\alpha\right)\subset\mathcal{F}^{\sharp}\left(\alpha,\psi\right)$
. To prove that $\mathcal{F}\left(\alpha,\psi\right)\subset R_{*}\left(\rho,\alpha\right)$,
let $z\in\mathcal{F}\left(\alpha,\psi\right)$. By definition of $\mathcal{F}\left(\alpha,\psi\right)$,
there exists $\omega\in I^{\N}$ such that $\left(\omega,z\right)\in\tilde{\mathcal{F}}\left(\alpha,\psi\right)$.
Hence, by Lemma \ref{lem:Q-via-ergodicsums}, we have that $\M\left(\rho,z\right)=\alpha$
and thus, $z\in R\left(\rho,\alpha\right)\subset R_{*}\left(\rho,\alpha\right)$.
To verify that $R_{*}\left(\rho,\alpha\right)\subset\mathcal{F}^{\sharp}\left(\alpha,\psi\right)$,
let $z\in R_{*}\left(\rho,\alpha\right)$, that is, $\M_{*}\left(\rho,z\right)=\alpha$.
Since \foreignlanguage{english}{$\rho$ is constant on each connected
component of $F\left(G\right)$ by \cite[Theorem 3.15 (1)]{s11random},
we have that $z\in J\left(G\right)$. By Proposition \ref{prop:basic-facts}
(\ref{enu:basic-facts-3}), there exists $\omega\in I^{\N}$ such
that $\left(\omega,z\right)\in J\left(\tilde{f}\right)$. }Lemma \ref{lem:Q-via-ergodicsums}
gives that $\liminf_{n}S_{n}\tilde{\psi}\left(\left(\omega,z\right)\right)\big/S_{n}\tilde{\zeta}\left(\left(\omega,z\right)\right)=\alpha$,
which implies that $\left(\omega,z\right)\in\tilde{\mathcal{F}}^{\sharp}\left(\alpha,\psi\right)$.
Hence, we have $z\in\mathcal{F}^{\sharp}\left(\alpha,\psi\right)$.
We have thus shown that $\mathcal{F}\left(\alpha,\psi\right)\subset R_{*}\left(\rho,\alpha\right)\subset\mathcal{F}^{\sharp}\left(\alpha,\psi\right)$
and the proof is complete. 
\end{proof}

\section{Examples\label{sec:Examples}}
We give some examples to which we can apply the main theorems. 
\begin{enumerate}
\item Let $f=\left(f_{i}\right)_{i\in I}\in\left(\Rat\right)^{I}$
be expanding and let \foreignlanguage{english}{\textup{$G=\left\langle f_{i}:i\in I\right\rangle $.}}
Suppose that $f$ satisfies the separation condition. Let $\left(p_{i}\right)_{i\in I}\in\left(0,1\right)^{I}$
be a probability vector, let $\tau:=\sum_{i\in I}p_{i}\delta_{f_{i}}$. 
Suppose that $G$ has at least two minimal sets. 
Here, we say that a non-empty compact subset $L$ of $\hat{\C }$ is a minimal 
set of $G$ if $L$ is minimal among the space 
$\{ K\mid K \mbox{ is a non-empty compact subset of }\hat{\C } 
\mbox{ and }g(K)\subset K \mbox{ for each } g\in G\} $ with respect to the inclusion. 
(Note that if $K$ is a non-empty compact subset $\hat{\C }$ such that 
$g(K)\subset K$ for each $g\in G$, then by Zorn's lemma, there exists a minimal 
set $L$ of $G$ with $L\subset K.$) 
Let $T_{L,\tau }:\hat{\C }\rightarrow [0,1]$ be the function of probability of tending to $L$ which is defined as 
$$ T_{L,\tau }(z)=\left(\otimes _{n=1}^{\infty }\tau \right)
\left( \{ \omega =(\omega _{1},\omega _{2},\ldots )\in \{ f_{i}\mid i\in I\} ^{\N } 
\mid d(\omega _{n}\cdots \omega _{1}(z), L)\rightarrow 0 \mbox{ as }n\rightarrow \infty \} \right).$$ 
Then by \cite{s11random} $T_{L,\tau }$ is a non-constant function belonging to $U_{\tau }.$ 
In fact, $M_{\tau }(T_{L,\tau })=T_{L,\tau }.$ 
Thus, we can apply Theorem~\ref{thm:mf-for-hoelderexponent} to 
$f$ and $\rho =T_{L,\tau }.$ The function 
$T_{L,\tau }$ can be regarded as a complex analogue of the devil's staircase and 
Lebesgue's singular functions (see \cite[Introduction]{s11random}). 
If each $f_{i}$ is a polynomial and $L=\{ \infty \}$ then $T_{\infty ,\tau }:=T_{\{ \infty \} ,\tau }$ is sometimes called 
a devil's coliseum. 

Since $(f_{\omega })_{\omega \in I^{k}}$, where $k$ is a large positive integer, 
satisfies the open set condition, all statements in Theorem~\ref{thm:multifractalformalism} and 
Proposition~\ref{lem:degenerate-via-zdunik} hold for $f.$   
\item Let $f_{1}$, $f_{2}$ be two polynomials with $\deg\left(f_{i}\right)\ge2$
for $i\in\left\{ 1,2\right\} $. Let $I=\left\{ 1,2\right\} $ and
$G=\left\langle f_{i}:i\in I\right\rangle $. Let $\left(p_{1},p_{2}\right)\in\left(0,1\right)^{2}$
with $p_{1}+p_{2}=1$ and let $\tau=p_{1}\delta_{f_{1}}+p_{2}\delta_{f_{2}}.$
Suppose that $G$ is hyperbolic, $P\left(G\right)\setminus\left\{ \infty\right\} $
is bounded in $\C$ and that $J\left(G\right)$ is disconnected. Then
$f=\left(f_{1},f_{2}\right)$ is expanding, $f$ satisfies the separating condition (see
\cite[Theorem 1.7]{MR2553369}) and the function of probability of
tending to infinity $T_{\infty,\tau}=T_{\{ \infty \} ,\tau }:\Chat\rightarrow\left[0,1\right]$ 
is a non-constant function belonging to $U_{\tau }.$ 
To this $f$ and $\rho =T_{\infty ,\tau }$ we can apply Theorem~\ref{thm:mf-for-hoelderexponent}. 
To see a concrete example, let $g_{1}(z)=z^{2}-1, g_{2}(z)=z^{2}/4$, $f_{1}=g_{1}\circ g_{1}, 
f_{2}=g_{2}\circ g_{2}.$ Then $f=(f_{1},f_{2})$ satisfies the above condition (see \cite[Example 6.2]{s11random}).  
\item Let $f_{1}$ be a hyperbolic polynomial with $\deg\left(f_{1}\right)\ge2$.
Suppose that $J\left(f_{1}\right)$ is connected. Let $h\in\Int\left(K\left(f_{1}\right)\right)$.
Let $d\in\N$ with $d\ge2$ and $\left(\deg f_{1},d\right)\neq\left(2,2\right)$.
Then there exists a constant $a_{0}>0$ such that for each $a\in\C$
with $0<\left|a\right|<a_{0}$, setting $f_{2,a}\left(z\right):=a\left(z-h\right)^{d}+h$,
we have that (i) $G_{a}:=\left\langle f_{1},f_{2,a}\right\rangle $
is hyperbolic, (ii) $f_{a}=\left(f_{1},f_{2,a}\right)$ satisfies
the separating condition and (iii) $P\left(\left\langle f_{1},f_{2,a}\right\rangle \right)\setminus\left\{ \infty\right\} $
is bounded in $\C$ (see \cite[Proposition 6.1]{s11random}). 
\item There are many examples of $f=\left(f_{i}\right)_{i\in I}\in\Rat^{I}$,
which satisfy the assumptions of Theorem \ref{thm:mf-for-hoelderexponent}
(see \cite[Propositions 6.3, 6.4 and 6.5]{s11random}). 
\end{enumerate}
Finally we give an important remark on the estimate of $\alpha _{-}$ and the 
non-differentiability of non-constant $\rho \in U_{\tau }.$ 
\begin{rem}
Let $f=\left(f_{i}\right)_{i\in I}\in\Rat^{I}$ and suppose that each
$f_{i}$ is a polynomial with $\deg\left(f_{i}\right)\ge2$. Under
the assumptions of Theorem \ref{thm:mf-for-hoelderexponent}, we have
by Theorem \ref{thm:mf-for-hoelderexponent} and \cite[Theorem 3.82]{s11random}
that  
\[
\alpha_{-}\left(\psi\right)\le\frac{-\sum_{i\in I}p_{i}\log p_{i}}{\sum_{i\in I}p_{i}\log\deg f_{i}+\int_{\Gamma^{\N}}\sum_{c}\mathcal{G}_{\gamma}\left(c\right)d\hat{\tau}\left(\gamma\right)}\le\alpha_{+}\left(\psi\right),
\]
where $\Gamma=\left\{ f_{i}:i\in I\right\} $, $\hat{\tau}=\bigotimes_{j=1}^{\infty}\tau$,
and $\mathcal{G}_{\gamma}$ denotes the Green's function of the basin
$A_{\infty,\gamma}$ of $\infty$ for the sequence $\gamma=\left(\gamma_{1},\gamma_{2},\dots\right)\in\Gamma^{\N}$
and $c$ runs over the critical points of $\gamma_{1}$ in $A_{\infty,\gamma}$.
For the details we refer to \cite[Theorem 3.82]{s11random}.

Moreover, in addition to the assumptions of Theorem \ref{thm:mf-for-hoelderexponent},
if each $f_{i}$ is a polynomial with $\deg(f_{i})\ge2$ and if (a)
$\sum_{i\in I}p_{i}\log\left(p_{i}\deg f_{i}\right)>0$ or (b) $P\left(G\right)\setminus\left\{ \infty\right\} $
is bounded in $\C$ or (c) $\card\left(I\right)=2$, then 
\[
\alpha_{-}\left(\psi\right)\le\frac{-\sum_{i\in I}p_{i}\log p_{i}}{\sum_{i\in I}p_{i}\log\deg f_{i}+\int_{\Gamma^{\N}}\sum_{c}\mathcal{G}_{\gamma}\left(c\right)d\hat{\tau}\left(\gamma\right)}<1.
\]
See \cite[Theorem 3.82]{s11random}. For the proof, we use potential
theory.
\end{rem}
\newcommand{\etalchar}[1]{$^{#1}$}
\def\cprime{$'$}
\providecommand{\bysame}{\leavevmode\hbox to3em{\hrulefill}\thinspace}
\providecommand{\MR}{\relax\ifhmode\unskip\space\fi MR }
\providecommand{\MRhref}[2]{%
  \href{http://www.ams.org/mathscinet-getitem?mr=#1}{#2}
}
\providecommand{\href}[2]{#2}

\end{document}